\documentclass[aop,MSNbibl,dvips]{arximspdf}
\usepackage{mathrsfs}

\doi{10.1214/12-AOP795} 
\volume{41}
\issue{4}
\pubyear{2013}
\firstpage{2513}
\lastpage{2543}

\makeatletter

\newcommand{\rright}{\right}
\newcommand{\lleft}{\left}

\newtheorem{theorem}{Theorem}
\newtheorem{corollary}{Corollary}
\newtheorem{lemma}{Lemma}
\newtheorem{proposition}{Proposition}

\newproclaim{remark}{Remark}

\newcommand{\Tr}{\operatorname{Tr}}
\newcommand{\Det}{\operatorname{Det}}

\makeatother

\begin{document}
\begin{frontmatter}

\title{Determinantal point processes with $J$-Hermitian
correlation kernels}
\runtitle{Determinantal point processes}

\begin{aug}
\author[A]{\fnms{Eugene} \snm{Lytvynov}\corref{}\thanksref{t1}\ead[label=e1]{e.lytvynov@swansea.ac.uk}\ead[url,label=u1]{http://www-maths.swan.ac.uk/staff/ewl/}}
\runauthor{E. Lytvynov}
\affiliation{Swansea University}
\address[A]{Department of Mathematics\\
Swansea University\\
Singleton Park\\
Swansea SA2 8PP\\
United Kingdom\\
\printead{e1}\\
\printead{u1}} 
\end{aug}

\thankstext{t1}{Supported in part by the SFB 701 ``Spectral
structures and topological methods in mathematics,'' Bielefeld
University and by International Joint Project Grant 2008/R2
of the Royal Society.}

\received{\smonth{4} \syear{2011}}
\revised{\smonth{8} \syear{2012}}

%
\begin{abstract}
Let $X$ be a locally compact Polish space and let $m$ be a reference
Radon measure on $X$. Let $\Gamma_X$ denote the configuration space
over $X$, that is, the space of all locally finite subsets of $X$. A
point process on $X$ is a probability measure on $\Gamma_X$. A point
process $\mu$ is called determinantal if its correlation functions have
the form $k^{(n)}(x_1,\ldots,x_n)=\det[K(x_i,x_j)]_{i,j=1,\ldots,n} $.
The function $K(x,y)$ is called the correlation kernel of the
determinantal point process $\mu$. Assume that the space $X$ is split
into two parts: $X=X_1\sqcup X_2$. A kernel $K(x,y)$ is called
$J$-Hermitian if it is Hermitian on $X_1\times X_1$ and $X_2\times
X_2$, and $K(x,y)=-\overline{K(y,x)}$ for $x\in X_1$ and $y\in X_2$. We
derive a necessary and sufficient condition of existence of a
determinantal point process with a $J$-Hermitian correlation kernel
$K(x,y)$.
\end{abstract}

%

\end{frontmatter}

\section{Introduction and preliminaries}\label{vhytd}
\subsection{Macchi--Soshnikov theorem}\label{kgfyufr}

Let $X$ be a locally compact Polish space, let $\mathcal B(X)$ be the
Borel $\sigma$-algebra on $X$, and let $\mathcal B_0(X)$ denote the
collection of all sets from $\mathcal B(X)$ which are pre-compact. The
configuration space over $X$ is defined as the set of all locally
finite subsets of $X$:
\[
\Gamma:=\Gamma_X:=\bigl\{\gamma\subset X\mid\mbox{for all $\Delta
\in\mathcal B_0(X)$ $|\gamma\cap\Delta|<\infty$}\bigr\}.
\]
Here, for a set $\Lambda$, $|\Lambda|$ denotes its capacity.
Elements $\gamma\in\Gamma$ are called configurations. The space
$\Gamma$ can be endowed with the vague topology, that is, the weakest
topology on $\Gamma$ with respect to which all maps $\Gamma\ni\gamma
\mapsto\sum_{x\in\gamma} f(x)$, $f\in C_0(X)$, are continuous. Here
$C_0(X)$ is the space of all continuous real-valued functions on $X$
with compact support. The configuration space $\Gamma$ equipped with
the vague topology is a Polish space.
We will denote by $\mathcal B(\Gamma)$ the Borel $\sigma$-algebra on
$\Gamma$.
A probability measure $\mu$ on $(\Gamma,\mathcal B(\Gamma))$ is
called a point process on $X$. For more detail, see, for example, \cite
{DVJ,KK,Lenard,KMM}.

A point process $\mu$ can be described with the help of correlation
functions, if they exist. Let $m$ be a reference Radon measure on
$(X,\mathcal B(X))$. The $n$th correlation function of $\mu$ ($n\in
\mathbb N$) is an $m^{\otimes n}$-a.e. nonnegative measurable
symmetric function $k_\mu^{(n)}(x_1,\ldots,x_n)$ on $X^n$ such that,
for any measurable symmetric function $f^{(n)}\dvtx X^n\to[0,\infty]$,
%
\begin{eqnarray}\label{cdtrs}
&&\int_\Gamma\sum_{\{x_1,\ldots,x_n\}\subset\gamma}f^{(n)}(x_1,\ldots,x_n) \mu(d\gamma)
\nonumber\\[-8pt]\\[-8pt]
&&\qquad=\frac1{n!}\int_{X^n}f^{(n)}(x_1,\ldots,x_n)k_\mu^{(n)}(x_1,\ldots,x_n) m(dx_1)\cdots m(dx_n).
\nonumber
\end{eqnarray}
Under a mild condition on the growth of correlation functions as $n\to
\infty$, they determine a point process uniquely~\cite{Lenard}.

A point\vspace*{1pt} process $\mu$ is called determinantal if there exists a
function $K(x,y)$ on~$X^2$, called the \textit{correlation kernel}, such that
%
\begin{equation}
\label{bgygfyuf} k_\mu^{(n)}(x_1,\ldots,x_n)=\det\bigl[K(x_i,x_j)
\bigr]_{i,j=1}^n,\qquad n\in\mathbb N;
\end{equation}
see, for example,~\cite{Soshnikov}. The integral operator $K$ in
$L^2(X,m)$ which has integral kernel $K(x,y)$ is called the
\textit{correlation operator of} $\mu$.

Assume that the correlation operator $K$ is self-adjoint and bounded on
the (real or complex) Hilbert space $L^2(X,m)$. In particular, the
integral kernel $K(x,y)$ is Hermitian (symmetric in the real case). If
the correlation functions $(k_\mu^{(n)})_{n\in\mathbb N}$ in (\ref
{bgygfyuf}) are \textit{pointwisely} nonnegative, then $K(x,y)$ is a
positive definite kernel. Hence, if additionally the function $K(x,y)$
is continuous (it being possible to weaken the latter condition), then
the operator $K$ must be nonnegative ($K\ge0$).

A bounded linear operator $K$ on $L^2(X,m)$ is called a locally
trace-class operator if, for each $\Delta\in\mathcal B_0(X)$, the
operator $K^\Delta:=P^\Delta KP^\Delta$ is trace-class. Here
$P^\Delta$ denotes the operator of multiplication by $\chi_\Delta$,
the indicator function of the set $\Delta$. [Thus, $P^\Delta$ is the
orthogonal projection of $L^2(X,m)$ onto $L^2(\Delta,m)$.] If the
operator $K$ is self-adjoint and nonnegative, then we can and will
choose its integral kernel, $K(x,y)$, so that
\[
\Tr K^\Delta=\int_\Delta K(x,x) m(dx) \qquad\mbox{for each }
\Delta\in\mathcal B_0(X);
\]
see~\cite{Soshnikov} and~\cite{GY}. By (\ref{cdtrs}) and (\ref
{bgygfyuf}), for each $\Delta\in\mathcal B_0(X)$,
\[
\int_\Gamma|\gamma\cap\Delta| \mu(d\gamma)=\int
_\Delta K(x,x) m(dx).
\]
Hence, in order that the correlation functions of $\mu$ be finite, we must
indeed assume that the operator $K$ is locally trace-class.

The following theorem, which is due to Macchi~\cite{Macchi} and
Soshnikov~\cite{Soshnikov}, plays a fundamental role in the theory of
point processes.

\begin{theorem}[(Macchi--Soshnikov)] Let $K$ be a self-adjoint,
nonnegative, locally trace-class, bounded linear operator on
$L^2(X,m)$. Then the integral kernel $K(x,y)$ of the operator $K$ is
the correlation kernel of a determinantal point process if and only if
$0\le K\le1$.
\end{theorem}

Note that, in the above theorem, the condition of boundedness of the
operator $K$ is not essential. One may instead initially assume that
$K$ is a Hermitian, nonnegative, locally trace-class operator which is
defined on a proper domain in $L^2(X,m)$.

Determinantal point processes with Hermitian correlation kernels occur
in various fields of mathematics and physics; see, for example, the
review paper~\cite{Soshnikov} and Chapter 4 in~\cite{AGZ}.

\subsection{Complementation principle (particle-hole duality)}

Assume that the underlying space $X$ is split into two disjoint parts:
$X=X_1 \sqcup X_2$.
Hence, we get $L^2(X,m)=L^2(X_1,m)\oplus L^2(X_2,m)$. For $i=1,2$, let
$P_i$ denote the orthogonal projection of $L^2(X,m)$ onto $L^2(X_i,m)$.
Let us define a bounded linear operator $J$ on $L^2(X,m)$ by $J:=P_1-P_2$.
Following, for example,~\cite{indefinite}, we define an (indefinite)
$J$-scalar product on $L^2(X,m)$ by
\[
[f,g]:=(Jf,g)=(P_1f,P_1g)-(P_2f,P_2g),\qquad
f,g\in L^2(X,m).
\]
Here $(\cdot,\cdot)$ denotes the usual scalar product in $L^2(X,m)$.
A bounded linear operator $K$ on $L^2(X,m)$ is called $J$-self-adjoint
if $[Kf,g]=[f,Kg]$ for all $f,g\in L^2(X,m)$.
An integral kernel $K(x,y)$ of a $J$-self-adjoint, integral operator
$K$ is called $J$-Hermitian. More precisely, $K(x,y)$ is $J$-Hermitian
if $K(x,y)=\overline{K(y,x)}$ if $x,y\in X_1$ or $x,y\in X_2$, and
$K(x,y)=-\overline{K(y,x)}$ if $x\in X_1$, $y\in X_2$.

For a bounded linear operator $K$ on $L^2(X,m)$, we denote
%
\begin{equation}
\label{cydyr} \widehat K:=KP_1+(1-K)P_2.
\end{equation}
As is easily seen, $K$ is $J$-self-adjoint if and only if $\widehat K$ is
self-adjoint.

Assume now that the underlying space $X$ is discrete, that is, $X$ is a
countable set, and as a topological space $X$ it totally disconnected.
Thus, a configuration $\gamma$ in $X$ is an arbitrary subset of $X$.
Let $m$ be the counting measure on $X$: $m(\{x\})=1$ for each $x\in X$.
Any linear operator $K$ in $L^2(X,m)$ may be identified with its matrix
$[K(x,y)]_{x,y\in X}$ [$K(x,y)$ being the integral kernel of $K$ in
this case].

Let $\mu$ be a point process on $X$. By (\ref{cdtrs}),
\[
k_\mu^{(n)}(x_1,\ldots,x_n)=\mu
\bigl(\gamma\in\Gamma\dvtx  \{x_1,\ldots,x_n\} \subset\gamma
\bigr)\vadjust{\goodbreak}
\]
for distinct points $x_1,\ldots,x_n\in X$, otherwise $k_\mu
^{(n)}(x_1,\ldots,x_n)=0$. In particular, the correlation functions
uniquely identify the corresponding point process.

Following~\cite{BOO}, we will now present a complementation principle
(a particle-hole duality) for determinantal point processes.
(This observation is referred by the authors of~\cite{BOO} to a
private communication by S. Kerov.) Assume, as above, that the
underlying space $X$ is divided into two disjoint parts: $X=X_1 \sqcup X_2$.
Consider the mapping $I\dvtx \Gamma\to\Gamma$ defined by
\[
I\gamma:=\widehat\gamma:=(\gamma\cap X_1)\cup(X_2
\setminus\gamma).
\]
Thus, on the $X_1$ part of the space, the configuration $\widehat
\gamma$ coincides with $\gamma$, while on the $X_2$ part the
configuration $\widehat\gamma$ consists of all points from $X_2$
which do not belong to $\gamma$ (holes). The mapping $I$ is clearly an
involution, that is, $I^2$ is the identity mapping. For a point process
$\mu$ on $X$, we denote by $\widehat\mu$ the push-forward of $\mu$
under $I$.

\begin{proposition}[(\cite{BOO})]\label{nfutdeseas} Let $\mu$ be an
arbitrary determinantal point process on a discrete space $X=X_1\sqcup
X_2$, with a correlation kernel
$K(x,y)$. Then $\widehat\mu$ is the determinantal point process on
$X$ with the correlation kernel $\widehat K(x,y)$, the integral kernel
of the operator $\widehat K$ defined by (\ref{cydyr}).
\end{proposition}

Combining the Macchi--Soshnikov theorem with Proposition \ref
{nfutdeseas}, we immediately get the following:

\begin{proposition}\label{ydst} Let $X=X_1\sqcup X_2$ be a discrete
space and let $m$ be a counting measure on $X$. Let $K$ be a bounded
linear operator on $L^2(X,m)$ and let $K$ be $J$-self-adjoint.
Then $K(x,y)$ is the correlation kernel of a determinantal point
process on $X$ if and only if $0\le\widehat K\le1$.
\end{proposition}

\subsection{Formulation of the problem and the main result}

In the case of a discrete underlying space $X$, determinantal point
processes with $J$-Hermitian correlation kernels occurred in Borodin
and Olshanski's studies on harmonic analysis of both the infinite
symmetric group and the infinite-dimensional unitary group; see, for
example,~\cite{BO1,BO2,BO3,BO4,O} and the references therein. The
paper~\cite{BO3}, page 1332, also contains references to some earlier
works of mathematical physicists on solvable models of systems with
positive and negative charged particles. In these papers, one finds
further examples of determinantal point processes with $J$-Hermitian
correlation kernels.

Furthermore, in their studies, Borodin and Olshanski derived three
classes of determinantal point processes with $J$-Hermitian correlation
kernels in the case where the underlying space $X$ is \textit{continuous}:
the Whittaker kernel~\cite{BO2} (\mbox{$X=\mathbb R_-\sqcup\mathbb R_+$}),
its scaling limit---the matrix tail kernel~\cite{O}
($X=\mathbb R\sqcup\mathbb R$), and the continuous hypergeometric\vadjust{\goodbreak}
kernel~\cite{BO3} [$X=(-\frac12,\frac12)\sqcup\{x\in\mathbb R\dvtx
|x|>\frac12\}$].
It is important to note that, in all these examples, the self-adjoint
operator $\widehat K$ appears to be an orthogonal projection. This
follows from Proposition~5.1 in~\cite{BO4} and the respective results
of \mbox{\cite{BO2,O}} (see also~\cite{BO4}, Proposition~6.6) and
\cite{BO3}. (It should be, however, noted that, in the case of a
continuous hypergeometric kernel, the corresponding projection property
was proved only under an additional assumption; see the last two
paragraphs of Section~10 in~\cite{BO3}.)

The aim of the present paper is to \textit{derive}, \textit{in the case of a
general underlying space $X$}, \textit{a necessary and sufficient condition of
existence of a determinantal point process with a $J$-Hermitian
correlation kernel}. This problem was formulated to the author by
Grigori Olshanski. I am extremely grateful to him for this and for many
useful discussions and suggestions.

Our main result may be stated as follows. (We will omit a technical
detail related to the choice of an integral kernel of the operator
$K$.)\vspace*{9pt}

\textit{Main result.} {Assume that $K$ is a $J$-self-adjoint bounded
linear operator on $L^2(X,m)$.
Assume that the operators $P_1KP_1$ and $P_2KP_2$ are nonnegative.
Assume that, for any $\Delta_1, \Delta_2\in\mathcal B_0(X)$ such
that $\Delta_1\subset X_1$ and $\Delta_2\subset X_2$,
the operators $ K^{\Delta_i}$ ($i=1,2$) are trace-class,
while $P^{\Delta_2}KP^{\Delta_1}$ is a Hilbert--Schmidt operator.
Then the integral kernel $K(x,y)$ of the operator $K$ is the
correlation kernel of a determinantal point process if and only if
$0\le\widehat K\le1$.}\vspace*{9pt}

Let us make two remarks regarding the conditions of the main result.
First, we note that, if the correlation operator $K$ of $\mu$ is
$J$-self-adjoint, then the restrictions of the point process $\mu$ to
$X_1$ and $X_2$ are determinantal point processes on $X_1$ and $X_2$
with self-adjoint, correlation operators $P_1KP_1$ and $P_2KP_2$,
respectively. Therefore, we assume that the latter operators are nonnegative.

Second, choose any $\Delta\in\mathcal B_0(X)$ such that $m(\Delta
_i)>0$, where $\Delta_i:=\Delta\cap X_i$, $i=1,2$. Then, since the
operator $K$ is not self-adjoint, the assumption in the main result is
weaker than the requirement that the operator $K^\Delta$ be
trace-class. In fact, $K$ being locally trace-class seems to be a
rather unnatural assumption for $J$-self-adjoint operators.
This, of course, leads us to some additional difficulties in the proof.

Clearly, Proposition~\ref{ydst} is the special case of our main result
in the case where the underlying space $X$ is discrete.
The drastic difference between the discrete and the continuous cases is
that the mapping $\gamma\mapsto\widehat\gamma$ has no analog in the
case of a continuous space $X$. Furthermore, if the space $X$ is not
discrete, the self-adjoint operator $\widehat K$
is not even an integral operator, so it cannot be a correlation
operator of a determinantal point process.

To prove the main result, we follow the strategy of dealing with
determinantal point processes through the corresponding Fredholm
determinants (compare with~\cite{Macchi,ST,Soshnikov}), or rather the
extension of Fredholm determinant as proposed in~\cite{BOO}.\vadjust{\goodbreak}

Combining the main result and Proposition 5.1 in~\cite{BO4},
we also derive a method of constructing a big class of determinantal
point processes with $J$-self-adjoint correlation operators $K$ such
that the corresponding operators $\widehat K$ are orthogonal
projections. This class includes
the above mentioned examples of determinantal point processes obtained
by Borodin and Olshanski.\looseness=-1

The paper is organized as follows. In Section~\ref{lkhgigtit} we prove
a couple of results related to the mentioned extension of the Fredholm
determinant. In Section~\ref{jgytdyd} we prove a series of auxiliary
statements regarding $J$-self-adjoint operators and their extended
Fredholm determinants. Finally, in Section~\ref{tyded6} we formulate
and prove the main results of the paper.

\section{An extension of the Fredholm determinant}\label{lkhgigtit}

We first recall the classical definition of a Fredholm determinant;
see, for example,~\cite{Simon} for further detail.
Let $H$ be a complex, separable Hilbert space with scalar product
$(\cdot,\cdot)$ and norm $\|\cdot\|$.
We denote by $\mathscr L(H)$ the space of all bounded linear operators
on $H$.
An operator $A\in\mathscr L(H)$ is called a trace-class operator if $\|
A\|_1=\Tr(|A|)<\infty$, where $|A|=(A^*A)^{1/2}$. The set of all
trace-class operators in $H$ will be denoted by $\mathscr L_1(H)$. The
trace of an operator $A\in\mathscr L_1(H)$ is given by $\Tr(A)=\sum
_{n=1}^\infty(Ae_n,e_n) $, where $\{e_n\}_{n=1}^\infty$ is an
orthonormal basis of $H$. $\Tr(A)$ is independent of the choice of a
basis. Note also that
$|{\Tr(A)}|\le\Tr(|A|)$. For any $A\in\mathscr L_1(H)$ and $B\in
\mathscr L(H)$, we have $AB, BA\in\mathscr L_1(H)$ with
\[
\max\bigl\{\|AB\|_1,\|BA\|_1\bigr\}\le\|A\|_1 \|B
\|,
\]
where $\|B\|$ denotes the usual operator norm of $B$. In the latter
case, we have
%
\begin{equation}
\label{Trace}\Tr(AB)=\Tr(BA).
\end{equation}

Denote by $\wedge^n(H)$ the $n$th antisymmetric tensor power of the
Hilbert space $H$, which is a closed subspace of $H^{\otimes n}$, the
$n$th tensor power of $H$. For any $A\in\mathscr L(H)$, the operator
$A^{\otimes n}$ in $H^{\otimes n}$ acts invariantly on $\wedge^n(H)$
and we denote by $\wedge^n(A)$ the restriction of $A^{\otimes n}$ to
$\wedge^n(H)$. If $A\in\mathscr L_1(H)$, then
$\wedge^n(A)\in\mathscr L_1(\wedge^n(H))$ and
%
\begin{equation}
\label{bvffdy} \bigl\|\wedge^n(A)\bigr\|_1\le\frac1{n!} \|A
\|_1^n.
\end{equation}
The Fredholm determinant is then defined by
%
\begin{equation}
\label{jdst}\Det(1+A)=1+\sum_{n=1}^\infty\Tr
\bigl(\wedge^n(A)\bigr).
\end{equation}
The Fredholm determinant can be characterized as the unique function
which is continuous in $A$ with respect to the trace norm $\|A\|_1$ and
which coincides with the usual determinant when $A$ is a
finite-dimensional operator.

One can extend the Fredholm determinant to a wider class on operators.
Assume that we are given a splitting of $H$ into two subspaces:
%
\begin{equation}
\label{nbytd}H=H_1\oplus H_2.\vadjust{\goodbreak}
\end{equation}
According to this splitting, we write an operator $A\in\mathscr L(H)$
in block form,
%
\begin{equation}
\label{fydydyf} A=\lleft[ %
\matrix{ A_{11}& A_{21}
\cr
A_{12}&A_{22} } %
\rright],
\end{equation}
where $A_{ij}\dvtx H_j\to H_i$, $i,j=1,2$. We define the even and odd parts
of $A$ as follows:
\[
A_{\mathrm{even}}:=\lleft[ %
\matrix{ A_{11}& 0
\cr
0&A_{22} } %
\rright],\qquad A_{\mathrm{odd}}:=\lleft[ %
\matrix{ 0& A_{21}
\cr
A_{12}&0 } %
\rright].
\]
We denote by $\mathscr L_{1|2}(H)$ the set of all operators $A\in
\mathscr L(H)$ such that $ A_{\mathrm{even}}\in\mathscr L_1(H)$ and $
A_{\mathrm{odd}}\in\mathscr L_2(H)$. Here $\mathscr L_2(H)$ denotes
the space of all Hilbert--Schmidt operators on $H$, equipped with the norm
\[
\|A\|_2= \Biggl(\sum_{n=1}^\infty
\|Ae_n\|^2 \Biggr)^{1/2},
\]
where $\{e_n\}_{n=1}^\infty$ is an orthonormal basis of $H$. Since
$\mathscr L_1(H)\subset\mathscr L_2(H)$, one concludes that
\[
\mathscr L_1(H)\subset\mathscr L_{1|2}(H)\subset\mathscr
L_{2}(H).
\]
We endow $\mathscr L_{1|2}(H)$ with the topology induced by the trace
norm on the even part and by the Hilbert--Schmidt norm on the odd part.

\begin{proposition}[(\cite{BOO})]\label{yrsdts}
The function $A\mapsto\Det(1+A)$ admits a unique extension to
$\mathscr L_{1|2}(H)$ which is continuous in the topology of $\mathscr
L_{1|2}(H)$. This extension is given by the formula
%
\begin{equation}\label{kfde5b}
\Det(1+A)=\Det\bigl((1+A)e^{-A}\bigr)\cdot e^{\Tr(A_{\mathrm
{even}})}.
\end{equation}
\end{proposition}

\begin{remark}
Note that, for each $A\in\mathscr L_2(H)$,
$(1+A)e^{-A}-1\in\mathscr L_1(H)$. Therefore, $\Det((1+A)e^{-A})$ is
a classical Fredholm determinant.
\end{remark}

\begin{remark} It should be noted that a possibility of extension of
the Fredholm determinant to $\mathscr L_{1|2}(H)$ was already known to
Berezin in the 1960s; see~\cite{Berezin}, page~8.
\end{remark}

We will now give another useful representation of $\Det(1+A)$ for
$A\in\mathscr L_{1|2}(H)$.

\begin{proposition}\label{kjlogt9i7tg} Let $A\in\mathscr L_{1|2}(H)$
have a block form (\ref{fydydyf}). Assume that $\|A_{11}\|<1$. Then
%
\begin{equation}
\label{yfr76r}\Det(1+A)=\Det(1+A_{11})\cdot\Det\bigl
(1+A_{22}-A_{21}(1+A_{11})^{-1}A_{12}
\bigr).
\end{equation}
[On the right-hand side of formula (\ref{yfr76r}), both factors are
\textit{classical} Fredholm determinants, as both operators $A_{11}$
and $A_{22}-A_{21}(1+A_{11})^{-1}A_{12}$ belong to $\mathscr L_1(H)$.]
\end{proposition}

\begin{remark}\label{uifr86} It should be stressed that the inequality
$\|A_{11}\|<1$ can be achieved by \textit{every} operator in $\mathscr
L_{1|2}(H)$. More generally, for each fixed $\varepsilon>0$, we can
always assume that $\|A_{11}\|<\varepsilon$. Indeed, assume $\|A_{11}\|
\ge\varepsilon$. By the canonical decomposition of a compact (in
particular, trace-class) operator (e.g.,~\cite{Simon}, Theorems 1.1
and 1.2), there exists an orthogonal splitting $H_1=H_1'\oplus
R$ such that the operator $A_{11}$ acts invariantly in both subspaces
$H_1'$ and~$R$, the subspace $R$ is finite-dimensional, and the norm of
the operator $A_{11}$ in the space $H_1'$ is strictly less than
$\varepsilon$. Setting $H_2':=H_2\oplus R$, we get a new orthogonal
splitting $H=H_1'\oplus H_2'$. Write the operator $A$ in the block form
with respect to this new splitting of $H$. Since $R$ is a
finite-dimensional space, the even part of $A$ in the new splitting is
still a trace-class operator, while the odd part of $A$ in the new
splitting is still a Hilbert--Schmidt operator.
\end{remark}

\begin{pf*}{Proof of Proposition~\ref{kjlogt9i7tg}}
For $i=1,2$, let $\{P_i^{(n)}\}_{n=1}^\infty$ be an ascending sequence
of finite-dimensional orthogonal projections in $H_i$ such that
$P_i^{(n)}$ strongly converges to the identity operator in $H_i$ as
$n\to\infty$.
Set $P^{(n)}:=P^{(n)}_1+P^{(n)}_2$, $n\in\mathbb N$. Then, for each
$n\in\mathbb N$, $A^{(n)}:=P^{(n)}AP^{(n)}$ is a finite-dimensional
operator in $H$,
and
\[
\bigl\|A^{(n)}-A\bigr\|_{1|2}\to0 \qquad\mbox{as }n\to\infty.
\]
Hence, by Proposition~\ref{yrsdts},
%
\begin{equation}
\label{s54w}\Det(1+A)=\lim_{n\to\infty}\Det\bigl(1+A^{(n)}\bigr).
\end{equation}
In the block form,
%
\begin{equation}
\label{hdydjgf} A^{(n)}=\lleft[ %
\matrix{A_{11}^{(n)}&A_{21}^{(n)}
\vspace*{2pt}\cr
A_{12}^{(n)}&A_{22}^{(n)} } %
\rright],
\end{equation}
where $A_{ij}^{(n)}=P_{i}^{(n)}A_{ij}P_{j}^{(n)}$, $i,j=1,2$.
For each $n\in\mathbb N$, the operator $A^{(n)}$ is
finite-dimensional, hence, $\Det(1+A^{(n)})$ is a classical Fredholm
determinant. Therefore,
%
\begin{equation}
\label{vgjfykt}\Det\bigl(1+A^{(n)}\bigr)= \Det\lleft[ %
\matrix{ 1+A_{11}^{(n)}&+A_{21}^{(n)}
\vspace*{2pt}\cr
A_{12}^{(n)}&1+A_{22}^{(n)} } %
\rright];
\end{equation}
the latter (in fact, usual) determinant refers to the
finite-dimensional Hilbert space $P^{(n)}H$.
Since $\|A_{11}\|<1$, we have $\|A_{11}^{(n)}\|<1$ for all $n$. Hence,
$1+A_{11}^{(n)}$ is invertible in $P_{11}^{(n)}H$. Employing the
well-known formula for the determinant of a block matrix, we get from
(\ref{s54w}) and (\ref{vgjfykt})
%
\begin{equation}\label{lgfwsi}
\Det(1+A) =\lim_{n\to\infty}\Det\bigl(1+A_{11}^{(n)}\bigr)
\cdot\Det\bigl(1+A_{22}^{(n)}-A_{21}^{(n)}
\bigl(1+A_{11}^{(n)}\bigr)^{-1}A_{12}^{(n)}
\bigr).\hspace*{-30pt}
\end{equation}
We state that
%
\begin{eqnarray}
\label{gfdh}
&\displaystyle \bigl\|A_{11}^{(n)}- A_{11}\bigr\|_1\to0,\qquad
\bigl\|A_{22}^{(n)}- A_{22}\bigr\|_1\to0,&
\\
\label{dsehgf}
&\displaystyle \bigl\|A_{21}^{(n)}\bigl(1-A_{11}^{(n)}
\bigr)^{-1}A_{12}^{(n)}- A_{21}(1-A_{11})^{-1}A_{12}
\bigr\|_1\to 0&
\end{eqnarray}
as $n\to\infty$. Formula (\ref{gfdh}) is evident. In view of the formula
\[
\|BC\|_1\le\|B\|_2\|C\|_2,\qquad B,C\in\mathscr
L_2(H)
\]
(see, e.g.,~\cite{Simon}, Theorem 2.8), the proof of (\ref{dsehgf})
is routine, so we skip it.
Thus, (\ref{yfr76r}) follows from (\ref{lgfwsi})--(\ref{dsehgf}).
\end{pf*}

We will now derive an analog of formula (\ref{jdst}) for $A\in
\mathscr L_{1|2}(H)$.
As follows from the proof of~\cite{ST}, Theorem 2.4, for each $A\in
\mathscr L_1(H)$, we have
%
\begin{equation}
\label{hfdrd} \Tr\bigl(\wedge^n(A)\bigr)=\frac1{n!}\sum
_{\xi\in S_n}\operatorname{sign}(\xi)\prod
_{\eta\in\operatorname{Cycle}(\xi)} \Tr\bigl(A^{|\eta|}\bigr).
\end{equation}
Here $S_n$ denotes the set of all permutations of $1,\ldots,n$, the
product in (\ref{hfdrd}) is over all cycles $\eta$ in permutation
$\xi$, and $|\eta|$ denotes the length of cycle $\eta$. For $A\in
\mathscr L_1(H)$, we clearly have $\Tr(A)=\Tr(A_{\mathrm{even}})$.
We further note that, for each $A\in\mathscr L_2(H)$, we have $A^k\in
\mathscr L_1(H)$ for $k\ge2$. Thus, for each $A\in\mathscr
L_{1|2}(H)$, we set $C_n(A)$ to be equal to the right-hand side of
(\ref{hfdrd}) in which we set
\[
\Tr(A):=\Tr(A_{\mathrm{even}}),\qquad A\in\mathscr L_{1|2}(H).
\]
Hence, $C_n(A)$ is well defined for each $A\in\mathscr L_{1|2}(H)$,
and $C_n(A)=\Tr(\wedge^n(A))$ for each $A\in\mathscr L_1(H)$.

\begin{proposition}\label{fdyddyr}
For each $A\in\mathscr L_{1|2}(H)$, we have
%
\begin{equation}
\label{f6e6}\Det(1+A)=1+\sum_{n=1}^\infty
C_n(A).
\end{equation}
\end{proposition}

\begin{pf} We know that formula (\ref{f6e6}) holds for all $A\in
\mathscr L_1(H)$. Next, for each $A\in\mathscr L_2(H)$,
\[
\bigl\|A^k\bigr\|_1 \le\| A\|^{k-2}\bigl\|A^2
\bigr\|_1\le\| A\|^{k-2}\|A\|_2^2\le\|A
\|_2^k,\qquad k\ge2.
\]
Hence, by the definition of $C_n(A)$,
%
\begin{equation}
\label{fd66yu}\bigl|C_n(A)\bigr|\le\|A\|_{1|2}^n,
\end{equation}
where
\[
\|A\|_{1|2}:=\max\bigl\{\|A\|_2,\|A_{\mathrm{even}}
\|_1\bigr\}.
\]
[Note that $\|\cdot\|_{1|2}$ is a norm on $\mathscr L_{1|2}(H)$ which
determines its topology.] Hence, if $\|A\|_{1|2}<1$, the series on the
right-hand side of (\ref{f6e6})\vadjust{\goodbreak} converges absolutely.
We fix any $A\in\mathscr L_{1|2}(H)$ with $\|A\|_{1|2}<1$.
For $i=1,2$, let $\{P_i^{(k)}\}_{k=1}^\infty$ be an ascending sequence
of finite-dimensional orthogonal projections as in the proof of
Proposition~\ref{kjlogt9i7tg}. Then, for each $k\in\mathbb N$,
$A^{(k)}:=P^{(k)}AP^{(k)}$ is a finite-dimensional operator in~$H$,
and
\[
\bigl\|A^{(k)}-A\bigr\|_{1|2}\to0 \qquad\mbox{as }k\to\infty.
\]
Hence, by (\ref{fd66yu}) and the dominated convergence theorem,
\[
\Det\bigl(1+A^{(k)}\bigr)=1+\sum_{n=1}^\infty
C_n\bigl(A^{(k)}\bigr)\to1+\sum
_{n=1}^\infty C_n(A).
\]
Therefore, by Proposition~\ref{yrsdts}, formula (\ref{f6e6}) holds in
this case.

Now we fix an arbitrary $A\in\mathscr L_{1|2}(H)$. Then, by (\ref
{fd66yu}), the function
\[
z\mapsto1+\sum_{n=1}^\infty
C_n(zA)=1+\sum_{n=1}^\infty
z^n C_n(A)
\]
is analytic on the set $\{z\in\mathbb C\dvtx  |z|<\|A\|_{1|2}^{-1}\}$.
Thus, by the uniqueness of analytic continuation, to prove the
proposition, it suffices to show that the function
\[
\mathbb C\ni z\mapsto\Det(1+zA)
\]
is entire.
But this can be easily deduced from Proposition~\ref{kjlogt9i7tg} and
Remark~\ref{uifr86}.
\end{pf}

Let us now assume that $H=L^2(X,m)$, where $X$ is a locally compact
Polish space and $m$ is a Radon measure on $(X,\mathcal B(X))$. We fix
any $X_1,X_2\in\mathcal B(X)$ such that $X=X_1\sqcup X_2$.
By
setting $H_i:=L^2(X_i,m)$, $i=1,2$, we get a splitting (\ref{nbytd})
of~$H$.

\begin{proposition}\label{fda}
Let $K\in\mathscr L_{1|2}(L^2(X,m))$ be an integral operator with
integral kernel $K(x,y)$ such that $\int_X|K(x,x)| m(dx)<\infty$ and
%
\begin{equation}
\label{fydex} \Tr(K_{\mathrm{even}})=\int_{X}K(x,x) m(dx).
\end{equation}
Then
%
\begin{equation}
\label{lkds}\Det(1+K)=1+\sum_{n=1}^\infty\frac
1{n!} \int_{X^n}\det\bigl[K(x_i,x_j)
\bigr]_{i,j=1,\ldots,n} m(dx_1)\cdots m(dx_n).\hspace*{-20pt}
\end{equation}
\end{proposition}

\begin{pf} For each $l=2,3,\ldots\,$, we have
%
\begin{equation}
\label{vtrsw} \Tr\bigl(K^l\bigr)=\int_{X^l}K(x_1,x_2)K(x_2,x_3)
\cdots K(x_l,x_1) m(dx_1)\cdots
m(dx_l).
\end{equation}
Note that the integral in (\ref{vtrsw}) is independent of the choice
of a version of the integral kernel of $K$. Hence, by the definition of
$C_n(K)$ and formulas (\ref{fydex}) and (\ref{vtrsw}), we conclude that
\[
C_n(K)=\frac1{n!}\int_{X^n}\det
\bigl[K(x_i,x_j)\bigr]_{i,j=1,\ldots,n}
m(dx_1)\cdots m(dx_n).
\]
Now formula (\ref{lkds}) follows from Proposition~\ref{fdyddyr}.
\end{pf}

\section{$J$-self-adjoint operators}\label{jgytdyd}

We again assume that a Hilbert space $H$ is split into two subspaces;
see (\ref{nbytd}). According to this splitting, we write a vector
$f\in H$ as $f=(f_2,f_2)$ and an operator $A\in\mathscr L(H)$ in the
block form~(\ref{fydydyf}).
Denote by $P_1$ and $P_2$ the orthogonal projections of the Hilbert
space $H$ onto $H_1$ and $H_2$, respectively.
Setting $J:=P_1-P_2$, we introduce an (indefinite) $J$-scalar product
on $H$ by
\[
[f,g]:=(Jf,g)=(f_1,g_1)-(f_2,g_2),\qquad
f,g\in H.
\]
An operator $A\in\mathscr L(H)$ is called self-adjoint in the
indefinite scalar product $[\cdot,\cdot]$, or $J$-self-adjoint, if
\[
[Af,g]=[f,Ag],\qquad f,g\in H;
\]
see, for example,~\cite{indefinite}. In terms of the block form
(\ref{fydydyf}), an operator $A\in\mathscr L(H)$ is $J$-self-adjoint
if and only if
%
\begin{equation}
\label{fdtsw} A_{11}^*=A_{11},\qquad A_{22}^*=A_{22},\qquad
A_{21}^*=-A_{12}.
\end{equation}

\begin{remark}\label{kdsea}
Assume $A$ is a usual matrix which has a block form (\ref{fydydyf}).
If the blocks of $A$ satisfy (\ref{fdtsw}), then we will call $A$ a
$J$-Hermitian matrix.
\end{remark}

For any $A\in\mathscr L(H)$, we denote by $\widehat A$ the operator
from $\mathscr L(H)$ given by
%
\begin{equation}
\label{hsst}\widehat A:=AP_1+(1-A)P_2
\end{equation}
or, equivalently, in the block form,
\[
\widehat A=\lleft[ %
\matrix{ A_{11}&A_{21}
\cr
-A_{12}&1-A_{22} } %
\rright].
\]
Clearly, if the operator $A$ is self-adjoint, then $\widehat A$ is
$J$-self-adjoint, while if $A$ is $J$-self-adjoint, then $\widehat A$
is self-adjoint. Also $\hspace*{2pt}\widehat{\hspace*{-2pt}\widehat A }=A$.

We will use below the following results.

\begin{lemma}\label{fdyrdeycx}
Let $A\in\mathscr L(H)$ be $J$-self-adjoint. Then $\|A\|=\|\widehat
A-P_2\|$.
\end{lemma}

\begin{pf}
We have $A=\widehat A P_1+(1-\widehat A)P_2$, hence,
\[
A^*=P_1\widehat A+P_2(1-\widehat A).\vadjust{\goodbreak}
\]
Denote by $B_{A^*}$ the quadratic form on $H$ with generator $A^*$.
For any $f,g\in H$,
\begin{eqnarray*}
B_{A^*}(f,g)&=&\bigl(A^*f,g\bigr)
\\
&=&(\widehat A_{11}f_1,g_1)+(\widehat
A_{12}f_2,g_1)\\
&&{}+\bigl((1-\widehat
A_{22})f_2,g_2\bigr)+(-\widehat
A_{21}f_1,g_2).
\end{eqnarray*}
Denote $\tilde g=(g_1,-g_2)=(\tilde g_1,\tilde g_2)$. Then
\begin{eqnarray*}
B_{A^*}(f,g)&=&(\widehat A_{11}f_1,\tilde
g_1)_{H}+(\widehat A_{12}f_2,\tilde
g_1)_{H}+\bigl((\widehat A_{22}-1)f_2,
\tilde g_2\bigr)_{H}+(\widehat A_{21}f_1,
\tilde g_2)
\\
&=&(\widehat Af,\tilde g) -(f_2,\tilde g)
\\
&=&\bigl((\widehat A-P_2)f,\tilde g\bigr)
\\
&=&B_{\widehat A-P_2}(f,\tilde g).
\end{eqnarray*}
From here
\[
\|\widehat A-P_2\|=\|B_{\widehat A-P_2}\|=\|B_{A^*}\|=\bigl\|A^*
\bigr\|=\|A\|.
\]
\upqed\end{pf}

\begin{proposition}\label{kfdscd}
Let $A\in\mathscr L(H)$ be $J$-self-adjoint and assume that $0\le
\widehat A\le1$. Then $\|A\|\le1$.
\end{proposition}

\begin{pf}
By Lemma~\ref{fdyrdeycx}, it suffices to show that $\|\widehat A-P_2\|
\le1$. Note that $\widehat A-P_2$ is self-adjoint. For each $f\in H$,
\[
\bigl((\widehat A-P_2)f,f\bigr) =(\widehat Af,f)
-(f_2,f_2) \le(\widehat Af,f) \le(f,f).
\]
Hence, $ \widehat A-P_2\le1$.
Next,
\[
\bigl((\widehat A-P_2)f,f\bigr) =(\widehat Af,f)
-(f_2,f_2) \ge-(f_2,f_2) \ge
-(f,f),
\]
and so $ \widehat A-P_2\ge-1$. Thus, $-1\le\widehat A-P_2\le1$,
which implies the statement.
\end{pf}

\begin{proposition}\label{ctehswt}
Let $A\in\mathscr L(H)$ be $J$-self-adjoint and assume that $0\le
\widehat A\le1$. Then $\|A\|=1$ if and only if $\| A_{\mathrm{even}}\|=1$.
\end{proposition}

\begin{pf}
By Lemma~\ref{fdyrdeycx}, it suffices to prove that $\|\widehat A-P_2\|
=1$ if and only if
$\| A_{\mathrm{even}}\|=1$. Let us first assume that $\|\widehat
A-P_2\|=1$.

Since $0\le\widehat A\le1$, we have $0\le\widehat A_{11}\le1$ and
$0\le\widehat A_{22}\le1$. Hence, $0\le A_{11}\le1$ and $0\le
A_{22}\le1$, and so $0\le A_{\mathrm{even}}\le1$, which in turn
implies that $\|A_{\mathrm{even}}\|\le1$. We have to consider two cases.

\textit{Case} 1. $-1\in\sigma(\widehat A-P_2)$. [Here, $\sigma(B)$
denotes the spectrum of an operator $B\in\mathscr L(H)$.] Then there
exists a sequence $(f^{(n)})_{n=1}^\infty$ in $H$ such that $\|
f^{(n)}\| =1$ and
\[
\bigl((\widehat A-P_2)f^{(n)},f^{(n)}\bigr) \to-1.
\]
Since $(\widehat Af^{(n)},f^{(n)}) \ge0$ and $(P_2f^{(n)},f^{(n)}) \le
1$, we get
\[
\bigl(\widehat A f^{(n)},f^{(n)}\bigr) \to0,\qquad
\bigl\|f_2^{(n)}\bigr\| \to1.
\]
Hence, $f_1^{(n)}\to0$. From here
\[
\bigl(\widehat A_{11}f_1^{(n)},f_1^{(n)}
\bigr) +\bigl(\widehat A_{21}f_1^{(n)},f_2^{(n)}
\bigr) +\bigl(\widehat A_{12}f_2^{(n)},f_1^{(n)}
\bigr) \to0.
\]
Thus,
\[
\bigl(\widehat A_{22}f_2^{(n)},f_2^{(n)}
\bigr) \to0.
\]
Hence,
\[
\biggl( \widehat A_{22}\frac{f_2^{(n)}}{\|f_2^{(n)}\| }, \frac
{f_2^{(n)}}{\|f_2^{(n)}\| } \biggr)
\to0.
\]
Hence, $0\in\sigma(\widehat A_{22})$, and so $1\in\sigma(1-\widehat
A_{22})=\sigma(A_{22})$.

\textit{Case} 2. $1\in\sigma(\widehat A-P_2)$. Then there exists a
sequence $(f^{(n)})_{n=1}^\infty$ in $ H$ such that $\|f^{(n)}\| =1$ and
\[
\bigl((\widehat A-P_2)f^{(n)},f^{(n)}\bigr) \to1.
\]
Since
$(\widehat Af^{(n)},f^{(n)}) \le1$ and $(P_2f^{(n)},f^{(n)}) \ge0$,
we get
\[
\bigl(\widehat Af^{(n)},f^{(n)}\bigr) \to1,\qquad
\bigl\|f_2^{(n)}\bigr\| \to0.
\]
From here, analogously to the above, we conclude that $1\in\sigma
(\widehat A_{11})=\sigma(A_{11})$.

Thus, in both cases, we get $\| A_{\mathrm{even}}\|=1$.
By inverting the arguments, we conclude the inverse statement.
\end{pf}


\begin{proposition}\label{lkfdsa}
Let $A\in\mathscr L(H)$ be $J$-self-adjoint and let $A\in\mathscr
L_{1|2}(H)$. Assume that $\|A\|<1$ and $A_{11}\ge0$.
Then $\Det(1-A)>0$.
\end{proposition}

\begin{pf} Since $\|A\|<1$, we get $\|A_{11}\|<1$. Hence, by
formula (\ref{yfr76r}),
%
\begin{eqnarray}\label{vfdtrser}
\Det(1-A) &=&\Det(1-A_{11})\cdot\Det\bigl(1-A_{22}-A_{21}(1-A_{11})^{-1}A_{12}
\bigr)
\nonumber\\[-8pt]\\[-8pt]
&=&\Det(1-A_{11})\cdot\Det\bigl(1-A_{22}+A_{12}^*(1-A_{11})^{-1}A_{12}
\bigr).\nonumber
\end{eqnarray}
Note that both operators
$-A_{11}$ and $-A_{22}+ A_{12}^*(1-A_{11})^{-1}A_{12}$ are trace-class
and self-adjoint.
Since $\|A_{11}\|<1$, we get $\Det(1-A_{11})>0$. Further, $\|A_{22}\|
<1$ and, hence, there exists $\varepsilon>0$ such that $1-A_{22}\ge
\varepsilon1$. Clearly, since $A_{11}\ge0$,
\[
A_{12}^*(1-A_{11})^{-1}A_{12}\ge0,
\]
which implies
\[
1-A_{22}+A_{12}^*(1-A_{11})^{-1}A_{12}
\ge\varepsilon1.
\]
From here
\[
\Det\bigl(1-A_{22}+A_{12}^*(1-A_{11})^{-1}A_{12}
\bigr)>0,
\]
and the proposition is proven.
\end{pf}

\begin{proposition}\label{fyd6s6yc} Let $A\in\mathscr L_{1|2}(H)$ and
let $A$ be $J$-self-adjoint. Let $0\le\widehat A\le1$ and let $\|A\|
<1$. Let $L:=A(1-A)^{-1}$. Then $L$ is $J$-self-adjoint, $L\in\mathscr
L_{1|2}(H)$, and $L_{11}\ge0$, $L_{22}\ge0$.
\end{proposition}

\begin{pf} We have
$L=A+\sum_{n=2}^\infty A^n$, and
\[
\sum_{n=2}^\infty\bigl\|A^n
\bigr\|_1\le\| A\|_2^2\sum
_{n=0}^\infty\|A\|^n<\infty.
\]
Hence, $\sum_{n=2}^\infty A^n\in\mathscr L_1(H)$, so $L\in\mathscr
L_{1|2}(H)$.

Let us show that the operator $L$ is $J$-self-adjoint. For any $f,g\in
H$, we have
\[
(Lf,g)=\sum_{n=1}^\infty\bigl(A^nf,g
\bigr)=\sum_{n=1}^\infty\bigl(f,\bigl(A^*
\bigr)^ng\bigr)=\sum_{n=1}^\infty
\bigl(f,(A_{11}-A_{21}-A_{12}+A_{22})^ng
\bigr).
\]
Denoting
$A_{11}':=A_{11}$, $A_{22}':=A_{22}$, $A_{12}':=-A_{12}$, $A_{21}':=-A_{21}$,
we get
%
\begin{equation}
\label{vts} (Lf,g)=\sum_{n=1}^\infty\mathop{
\sum_{i_k,j_k=1,2}}_{k=1,\ldots,n} \bigl(f,A'_{i_1j_1}A'_{i_2j_2}
\cdots A'_{i_nj_n}g\bigr).
\end{equation}
Assume that $f=f_1\in H_1$, $g=g_1\in H_1$. Then, in the latter sum,
the terms, in which the number of the $A_{12}'$ operators is not equal
to the number of the $A'_{21}$ operators, are equal to zero. Hence,
%
\begin{eqnarray}\label{ds567uio}
(L_{11}f_1,g_1)&=&(Lf_1,g_1)
=\sum_{n=1}^\infty\mathop{\sum
_{i_k,j_k=1,2}}_{k=1,\ldots,n} (f_1,A_{i_1j_1}A_{i_2j_2}
\cdots A_{i_nj_n}g_1)
\nonumber\\[-8pt]\\[-8pt]
&=&\sum_{n=1}^\infty\bigl(f_1,A^ng_1
\bigr)=(f_1,Lg_1)=(f_1,L_{11}g_1).\nonumber
\end{eqnarray}
Thus, $L_{11}^*=L_{11}$. Analogously, $L_{22}^*=L_{22}$.

In the case where $f=f_1\in H_1$ and $g=g_2\in H_2$, those terms in the
sum in~(\ref{vts}), in which the number of the $A'_{21}$ operators is
not equal to the number of the $A'_{12}$ operators plus one, are equal
to zero. Hence, similar to (\ref{ds567uio}), we get
\[
(L_{21}f_1,g_2)=(f_1,-L_{12}g_2),
\]
so $L_{21}^*=-L_{12}$. Thus, $L$ is $J$-self-adjoint.\vadjust{\goodbreak}

Next, we will show that $L_{11}\ge0$.
Analogously to the proofs of Propositions~\ref{kjlogt9i7tg} and \ref
{fdyddyr}, we define operators $A^{(n)}$, $n\in\mathbb N$.
Thus, each $A^{(n)}$ is $J$-self-adjoint and
%
\begin{equation}
\label{se5}\bigl\|A^{(n)}\bigr\|\le\|A\|<1.
\end{equation}
Let $\widehat{A}{}^{(n)}$ denote the corresponding transformation of the
operator $A^{(n)}$ in the Hilbert space $P^{(n)}H$. Recalling
representation (\ref{hdydjgf}) of $A^{(n)}$, we thus get
\[
\widehat{A}{}^{(n)}=\lleft[ %
\matrix{A_{11}^{(n)}&A_{21}^{(n)}
\vspace*{2pt}\cr
-A_{12}^{(n)}&P_2^{(n)}-A_{22}^{(n)}
} %
\rright]=P^{(n)}\lleft[ %
\matrix{A_{11}&A_{21}
\cr
-A_{12}&1- A_{22}
} %
\rright] P^{(n)}=P^{(n)} \widehat A
P^{(n)}.
\]
Since $0\le\widehat A\le1$, we therefore conclude that $0\le\widehat
{A}{}^{(n)}\le1$. In particular, $A_{11}^{(n)}\ge0$.

We may assume that the dimension of the Hilbert space $P^{(n)}H$ is
$n$. Choose an orthonormal basis $(e^{(i)})_{i=1,\ldots,n}$ of
$P^{(n)}H$ such that $e^{(i)}\in P_1^{(n)}H$, $i=1,\ldots,k$, and
$e^{(i)}\in P_2^{(n)}H$, $i=k+1,\ldots,n$. In terms of this orthonormal
basis, we may treat the operator $A^{(n)}$ in $P^{(n)}H$ as an $n\times
n$ $J$-Hermitian matrix $[A^{(n)}_{ij}]_{i,j=1,\ldots,n}$.
Let
\[
X^{(n)}:=\{1,2,\ldots,n\},\qquad X_1^{(n)}:=\{1,2,\ldots,k
\},\qquad X_2^{(n)}:=\{k+1,k+2,\ldots,n\},
\]
so that $X^{(n)}=X_1^{(n)}\sqcup X_2^{(n)}$.
In view of Proposition~\ref{ydst}, there exists a determinantal point
process $\mu^{(n)}$ on $\Gamma_{X^{(n)}}$ with correlation kernel
\[
K^{(n)}(i,j):=A^{(n)}_{ij},\qquad i,j=1,\ldots,n.
\]

By Proposition~\ref{lkfdsa}, we have $\det(1-A^{(n)})>0$.
Let
\[
L^{(n)}:=A^{(n)}\bigl(1-A^{(n)}\bigr)^{-1}.
\]
We define a possibly signed measure $\rho^{(n)}$ on the configuration
space $\Gamma_{X^{(n)}}$ by setting
\[
\rho^{(n)}\bigl(\{\varnothing\}\bigr):=\Det\bigl(1-A^{(n)}\bigr)
\]
and for each nonempty configuration
$\{i_1,i_2,\ldots,i_m\}\in\Gamma_{X^{(n)}}$,
\[
\rho^{(n)}\bigl(\{i_1,i_2,\ldots,i_m
\}\bigr):=\det\bigl(1-A^{(n)}\bigr)\cdot\det\bigl(L^{(n)}(i_u,i_v)
\bigr)_{u,v=1,\ldots,m}.
\]
Analogously to the proof of Theorem~\ref{jkfds5t} below, we may show
that $\rho^{(n)}=\mu^{(n)}$. Hence, for each nonempty configuration
$\{i_1,i_2,\ldots,i_m\}\in\Gamma_{X^{(n)}}$,
\[
\det\bigl(L^{(n)}(i_u,i_v)
\bigr)_{u,v=1,\ldots,m}\ge0.
\]
In particular, for any nonempty configuration $\{i_1,i_2,\ldots,i_m\}
\in\Gamma_{X_1^{(n)}}$,
\[
\det\bigl(L_{11}^{(n)}(i_u,i_v)
\bigr)_{u,v=1,\ldots,m}\ge0.
\]
Hence, by the Sylvester criterion, $L_{11}^{(n)}\ge0$, and so
%
\begin{equation}
\label{ctst}\bigl(L^{(n)}_{11}f_1,f_1
\bigr)\ge0,\qquad f_1\in H_1.
\end{equation}
By (\ref{se5}) and the dominated convergence theorem, we get
%
\begin{eqnarray}\label{jeratyuk}
\lim_{n\to\infty} \bigl(L_{11}^{(n)}f_1,f_1
\bigr)&=&\lim_{n\to\infty} \bigl(L^{(n)}f_1,f_1
\bigr)
\nonumber
\\
&=& \lim_{n\to\infty} \sum_{l=1}^\infty
\bigl(\bigl(A^{(n)}\bigr)^lf_1,f_1
\bigr)
\nonumber\\[-8pt]\\[-8pt]
&=& \sum_{l=1}^\infty\lim_{n\to\infty}\bigl(
\bigl(A^{(n)}\bigr)^lf_1,f_1\bigr)
\nonumber
\\
&=& \sum_{l=1}^\infty\bigl(A^lf_1,f_1
\bigr)=(Lf_1,f_1)=(L_{11}f_1,f_1).\nonumber
\end{eqnarray}
Thus, by (\ref{ctst}) and (\ref{jeratyuk}), $(L_{11}f_1,f_1)\ge0$
for all $f_1\in H_1$. Analogously, we get \mbox{$L_{22}\ge0$}.
\end{pf}

The following statement about $J$-Hermitian matrices was proven in
\cite{O}.

\begin{proposition}[(\cite{O})]\label{huydsd} Assume that $A$ is a
$J$-Hermitian matrix and assume that its diagonal blocks, $A_{11}$,
$A_{22}$, are nonnegative definite. Then
$\det(A)\ge0$.
\end{proposition}

\begin{remark}
Note that the arguments we used in the proof of Proposition~\ref
{lkfdsa} are similar to the arguments Olshanski~\cite{O} used to prove
Proposition~\ref{huydsd}.
\end{remark}

From now on we will again assume that $H=L^2(X,m)$, where $X$ is a
locally compact Polish space, $m$ is a Radon measure on $(X,\mathcal
B(X))$, and $X_1,X_2\in\mathcal B(X)$ are such that $X=X_1\sqcup X_2$.
We also set $H_i:=L^2(X_i,m)$, $i=1,2$.
We further define
\[
\mathcal B(X_i):=\bigl\{\Lambda\in\mathcal B(X)\mid\Lambda\subset
X_i\bigr\}
\]
and, analogously, $\mathcal B_0(X_i)$, for $i=1,2$.

For $\Delta\in\mathcal B_0(X)$, we denote by $P^\Delta$ the
orthogonal projection of $L^2(X,m)$ onto $L^2(\Delta,m)$, that is, the
operator of multiplication by $\chi_\Delta$. For an operator $K\in
\mathscr L(L^2(X,m))$, we denote $K^\Delta:=P^\Delta K P^\Delta$.
We will say that an operator $K\in\mathscr L (L^2(X,m))$ is locally
trace-class on $X_1$ and $X_2$ if, for each $\Delta\in\mathcal
B_0(X_i)$, $i=1,2$, we have $ K^\Delta\in\mathscr L_1(L^2(X,m))$.

\begin{proposition}\label{fde6ewst}
Let $K\in\mathscr L (L^2(X,m))$ be $J$-self-adjoint and a locally
trace-class operator on $X_1$ and $X_2$, and let $0\le\widehat K\le
1$. Then, for each $\Delta\in\mathcal B_0(X)$, $K^\Delta\in\mathscr
L_{1|2}(L^2(X,m))$.\vadjust{\goodbreak}
\end{proposition}

\begin{pf} For each $\Delta_1\in\mathcal B_0(X_1)$, we have
%
\begin{equation}
\label{fd6s}P^{\Delta_1} \widehat KP^{\Delta_1}= K^{\Delta_1}\in
\mathscr L_1\bigl(L^2(X,m)\bigr).
\end{equation}
Since $\widehat K\ge0$, we get $P^{\Delta_1} \widehat KP^{\Delta
_1}\ge0$. Hence, by (\ref{fd6s}),
$\sqrt{\widehat K} P^{\Delta_1}\in\mathscr L_2(L^2(X,m))$. Next,
for each $\Delta_2\in\mathcal B_0(X_2)$,
\[
P^{\Delta_2}(1-\widehat K)P^{\Delta_2}=K^{\Delta_2}\in\mathscr
L_1\bigl(L^2(X,m)\bigr).
\]
Hence, analogously to the above, $\sqrt{1-\widehat K} P^{\Delta
_2}\in\mathscr L_2(L^2(X,m))$. From here
\begin{eqnarray*}
KP^{\Delta_1}&=& \widehat KP^{\Delta_1}=\sqrt{\widehat K} \sqrt{\widehat
K} P^{\Delta_1}\in\mathscr L_2\bigl(L^2(X,m)\bigr),
\\
KP^{\Delta_2}&=&(1-\widehat K)P^{\Delta_2}=\sqrt{1-\widehat K} \sqrt{1-
\widehat K} P^{\Delta_2}\in\mathscr L_2\bigl(L^2(X,m)
\bigr).
\end{eqnarray*}
Therefore, $K(P^{\Delta_1}+P^{\Delta_2})\in\mathscr L_2(L^2(X,m))$.
Thus, for each $\Delta\in\mathcal B_0(X)$, $KP^\Delta\in\mathscr
L_2(L^2(X,m))$, and so $K^\Delta\in\mathscr L_2(L^2(X,m))$.

By our assumption, for each $\Delta\in\mathcal B_0(X)$,
\begin{eqnarray*}
K_{\mathrm{even}}^\Delta&=&P^\Delta K_{\mathrm{even}}P^\Delta
=P^{\Delta_1}K_{11}P^{\Delta_1}+P^{\Delta_2}K_{22}P^{\Delta_2}\\
&=&K^{\Delta_1}+K^{\Delta_2}\in\mathscr L_1
\bigl(L^2(X,m)\bigr).
\end{eqnarray*}
(Here $\Delta_i:=\Delta\cap X_i$, $i=1,2$.) Thus, $K^\Delta\in
\mathscr L_{1|2}(L^2(X,m))$.
\end{pf}

\begin{proposition}\label{hcdrt}
Let $K\in\mathscr L(L^2(X,m))$ be $J$-self-adjoint, let
$K^\Delta\in
\mathscr L_{1|2}(L^2(X,m))$ for each $\Delta\in\mathcal B_0(X)$, and
let $K_{11}\ge0$, $K_{22}\ge0$.
Then $K$ is an integral operator and its integral kernel $K(x,y)$ can
be chosen so that:

\begin{longlist}
\item
The kernel $K(x,y)$ is $J$-Hermitian.

\item For $i=1,2$ and any $x_1,\ldots,x_n\in X_i$ ($n\in\mathbb
N$), the matrix
\[
\bigl[K(x_i,x_j) \bigr]_{i,j=1,\ldots,n}
\]
is nonnegative definite.

\item For each $\Delta\in\mathcal B_0(X)$,
%
\begin{equation}
\label{jkfds}\Tr\bigl(K_{\mathrm{even}}^\Delta\bigr)=\int
_\Delta K(x,x) m(dx).
\end{equation}
\end{longlist}
\end{proposition}

\begin{pf} For any $\Delta_1\in\mathcal B_0(X_1)$ and $\Delta_2\in
\mathcal B_0(X_2)$, $P^{\Delta_2}KP^{\Delta_1}$ is a
Hilbert--Schmidt operator, hence an integral operator. Therefore, we
can choose an integral kernel of $K_{21}$, which is a function
$K_{21}(x,y)$ on $X_2\times X_1$. We now define an integral kernel
$K_{12}(x,y)$ of the operator $K_{12}$ by setting
$K_{12}(x,y):=-\overline{K_{21}(y,x)}$ for $(x,y)\in X_1\times X_2$.
Next, the operators $K_{11}$ and $K_{22}$ are nonnegative, locally
trace-class operators. Hence, we can choose their integral kernels
according to~\cite{GY}, Lemma A.3; see also~\cite{LM}, Section 3. By
combining the integral kernels $K_{ij}(x,y)$, $i,j=1,2$, we obtain an
integral kernel $K(x,y)$ of $K$ with needed properties.
\end{pf}


From now on, for an operator $K$ as in Proposition~\ref{hcdrt}, we
will always assume that its integral kernel satisfies statements
(i)--(iii) of this proposition.

We denote by $B_0(X)$ the space of all measurable bounded real-valued
functions on $X$ with compact support.
For each $\varphi\in B_0(X)$, we preserve the notation $\varphi$ for
the bounded linear operator of multiplication by $\varphi$ in $L^2(X,m)$.

\begin{proposition} \label{bvgl}
Let $K\in\mathscr L(L^2(X,m))$ be $J$-self-adjoint, let \mbox{$K_{11}\ge0$},
$K_{22}\ge0$, and let $K^\Delta\in\mathscr L_{1|2}(L^2(X,m))$ for
each $\Delta\in\mathcal B_0(X)$.
Fix any $\Delta\in\mathcal B_0(X)$ and any $\varphi\in B_0(X)$ which
vanishes outside $\Delta$.
Then
$K^\Delta\varphi,\break \operatorname{sgn}(\varphi)\sqrt{|\varphi
|}K\sqrt{|\varphi|}\in\mathscr L_{1|2}(L^2(X,m)) $ and
%
\begin{eqnarray}\label{lufdesw6}
\Det\bigl(1+K^\Delta\varphi\bigr)&=&\Det\bigl(1+ \operatorname{sgn}(
\varphi)\sqrt{|\varphi|}K\sqrt{|\varphi|} \bigr)
\nonumber\\
&=&1+\sum_{n=1}^\infty\frac1{n!} \int
_{X^n}\varphi(x_1)\cdots\varphi(x_n)\\
&&\hspace*{63.5pt}{}\times
\det\bigl[K(x_i,x_j)\bigr]_{i,j=1,\ldots,n}
m(dx_1)\cdots m(dx_n).
\nonumber
\end{eqnarray}
\end{proposition}

\begin{pf} Since $K^\Delta\in\mathscr L_{2}(L^2(X,m))$, $K^\Delta
\varphi\in\mathscr L_{2}(L^2(X,m))$. Since $K_{\mathrm{even}}^\Delta
\in\mathscr L_1(L^2(X,m))$,
\[
\bigl(K^\Delta\varphi\bigr)_{\mathrm{even}}=K^\Delta_{\mathrm{even}}
\varphi\in\mathscr L_1\bigl(L^2(X,m)\bigr).
\]
Thus, $K^\Delta\varphi\in\mathscr L_{1|2}(L^2(X,m))$.

Denote $\psi_1:=\operatorname{sgn}(\varphi)\sqrt{|\varphi|}$ and
$\psi_2:=\sqrt{|\varphi|}$, $\psi_1,\psi_2\in B_0(X)$.
Since $\psi_1$ and $\psi_2$ vanish outside $\Delta$, we get
\[
\psi_1 K\psi_2=\psi_1 K^\Delta
\psi_2
\]
and, analogously to the above, we conclude that $\psi_1 K\psi_2\in
\mathscr L_{1|2}(L^2(X,m))$.

Since $K^\Delta_{\mathrm{even}}\in\mathscr L_1(L^2(X,m))$ and since
$\psi_1,\psi_2\in\mathscr L(L^2(X,m))$, by (\ref{Trace}),
%
\begin{eqnarray}\label{taq4}
\Tr\bigl(\bigl(K^\Delta\varphi\bigr)_{\mathrm{even}}\bigr)&=&\Tr
\bigl(K^\Delta_{\mathrm
{even}}\psi_2\psi_1\bigr)=
\Tr\bigl(\psi_1K^\Delta_{\mathrm{even}} \psi_2
\bigr)
\nonumber\\[-8pt]\\[-8pt]
&=&\Tr(\psi_1K_{\mathrm{even}} \psi_2)= \Tr\bigl((
\psi_1 K \psi_2)_{\mathrm{even}}\bigr).\nonumber
\end{eqnarray}
Next, for $l=2,3,\ldots\,$,
%
\begin{eqnarray}\label{ctsgdrd}
\Tr\bigl((\psi_1 K\psi_2)^l\bigr)&=&\Tr
\bigl(\psi_1 K^\Delta\varphi K^\Delta\varphi\cdots
K^\Delta\varphi K^\Delta\psi_2\bigr)
\nonumber\\[-8pt]\\[-8pt]
&=&\Tr\bigl(K^\Delta\varphi K^\Delta\varphi\cdots
K^\Delta\varphi K^\Delta\psi_2\psi_1
\bigr)=\Tr\bigl(\bigl(K^\Delta\varphi\bigr)^l\bigr).\nonumber
\end{eqnarray}
By (\ref{taq4}) and (\ref{ctsgdrd}), $C_n(K^\Delta\varphi)=C_n(\psi_1
K\psi_2)$ for each $n\in\mathbb N$, hence, formula (\ref
{lufdesw6}) holds.

Next, we note that the integral kernel $K^\Delta(x,y)$ of the operator
$K^\Delta$
is the restriction of $K(x,y)$ to $\Delta^2$. Clearly, the integral
kernel of $K^\Delta\varphi$ is $K^\Delta(x,y)\varphi(y)$. Using
(\ref{jkfds}), it is not hard to show that
\[
\Tr\bigl(\bigl(K^\Delta\varphi\bigr)_{\mathrm{even}}\bigr)=\int
_X K^\Delta(x,x)\varphi(x) m(dx).
\]
Hence, by Proposition~\ref{fda},
\begin{eqnarray*}
&&
\Det\bigl(1+K^\Delta\varphi\bigr)\\
&&\qquad=1+\sum_{n=1}^\infty
\frac1{n!} \int_{X^n}\det\bigl[K^\Delta(x_i,x_j)
\varphi(x_j)\bigr]_{i,j=1,\ldots,n} m(dx_1)\cdots
m(dx_n)
\\
&&\qquad=1+\sum_{n=1}^\infty\frac1{n!}\int
_{X^n}\varphi(x_1)\cdots\varphi(x_n)\\
&&\hspace*{63pt}\qquad\quad{}\times\det\bigl[K^\Delta(x_i,x_j)\bigr]_{i,j=1,\ldots,n}
m(dx_1)\cdots m(dx_n)
\\
&&\qquad=1+\sum_{n=1}^\infty\frac1{n!}\int
_{X^n}\varphi(x_1)\cdots\varphi(x_n)\\
&&\hspace*{63pt}\qquad\quad{}\times
\det\bigl[K(x_i,x_j)\bigr]_{i,j=1,\ldots,n}
m(dx_1)\cdots m(dx_n).
\end{eqnarray*}
\upqed\end{pf}

\section{Main results}\label{tyded6}

We again assume that $X$ is a locally compact Polish space and $m$ is a
Radon measure on $(X,\mathcal B(X))$. We will also assume that $m$
takes a positive value on each open nonempty set in $X$. Let
$\Gamma=\Gamma_X$ be the configuration space over $X$. Let $\mu$ be a
point process on $X$, that is, a probability measure on
$(\Gamma,\mathcal B(\Gamma))$. Assume that $\mu$ satisfies
%
\begin{equation}
\label{dsts} \int_\Gamma C^{|\gamma\cap\Delta|} \mu(d\gamma)
\qquad\mbox{for all $\Delta\in\mathcal B_0(X)$ and all $C>0$}.
\end{equation}
Then the Bogoliubov functional of $\mu$ is defined as
%
\begin{equation}
\label{vyds} B_\mu(\varphi):=\int_\Gamma\prod
_{x\in\gamma}\bigl(1+\varphi(x)\bigr) \mu(d\gamma),\qquad \varphi
\in B_0(X).
\end{equation}
Note that, since the function $\varphi$ has compact support, only a
finite number of terms in the product $\prod_{x\in\gamma}(1+\varphi
(x))$ are not equal to one. Note also that the integrability of the
function $\prod_{x\in\gamma}(1+\varphi(x))$ for each $\varphi\in
B_0(X)$ is equivalent to condition (\ref{dsts}).
If a point process $\mu$
has correlation functions $(k_\mu^{(n)})_{n=1}^\infty$ [see (\ref{cdtrs})],
then condition (\ref{dsts}) is also equivalent to
\begin{eqnarray}
\sum_{n=1}^\infty\frac{C^n}{n!}\int
_{\Delta^n}k_\mu^{(n)}(x_1,\ldots,x_n) m(dx_1)\cdots
m(dx_n)<\infty\nonumber\\
&&\eqntext{\mbox{for all $\Delta\in\mathcal B_0(X)$ and all $C>0$},}
\end{eqnarray}
and the Bogoliubov functional of $\mu$ is given by
%
\begin{equation}
\label{gsts}B_\mu(\varphi)=1+\sum_{n=1}^\infty
\frac1{n!}\int_{X^n}\varphi(x_1)\cdots
\varphi(x_n)k_\mu^{(n)}(x_1,\ldots,x_n) m(dx_1)\cdots m(dx_n)\hspace*{-30pt}
\end{equation}
for each $\varphi\in B_0(X)$.
The Bogoliubov functional of $\mu$ uniquely determines this point
process. For more detail about the Bogoliubov functional see, for
example,~\cite{BF}.

Let us now briefly recall some known facts about configuration spaces
and point processes; see, for example,~\cite{DVJ,KMM} for further details.
The $\sigma$-algebra $\mathcal B(\Gamma)$ coincides with the minimal
$\sigma$-algebra on $\Gamma$ with respect to which all mappings of
the form $\Gamma\ni\gamma\mapsto|\gamma\cap\Lambda|$ with
$\Lambda\in\mathcal B_0(X)$ are measurable. For a fixed set $\Delta
\in\mathcal B(X)$,
we denote by $\mathcal B_\Delta(\Gamma)$ the minimal $\sigma
$-algebra on $\Gamma$ with respect to which all mappings of the form
$\Gamma\ni\gamma\mapsto|\gamma\cap\Lambda|$ with $\Lambda\in
\mathcal B_0(X)$, $\Lambda\subset\Delta$, are measurable. In
particular, $\mathcal B_\Delta(\Gamma)$ is a sub-$\sigma$-algebra of
$\mathcal B(\Gamma)$. The $\sigma$-algebras $\mathcal B(\Gamma_\Delta)$
and $\mathcal B_\Delta(\Gamma)$ can be identified in the
sense that, for each $A\in\mathcal B(\Gamma_\Delta)$, $\{\gamma\in
\Gamma\mid\gamma\cap\Delta\in A\}\in\mathcal B_\Delta(\Gamma)$
and each set from $\mathcal B_\Delta(\Gamma)$ has a unique such
representation. Hence, the restriction of a point process $\mu$ on $X$
to the $\sigma$-algebra $\mathcal B_\Delta(\Gamma)$---denoted by
$\mu_\Delta$---can be identified with a point process on $\Delta$,
that is, with a probability measure on $(\Gamma_\Delta,\mathcal
B(\Gamma_\Delta))$.

Let $\Delta$ be a compact subset of $X$. Then the configuration space
$\Gamma_\Delta$ consists of all finite configurations in $\Delta$,
that is,
$\Gamma_\Delta=\bigsqcup_{n=0}^\infty\Gamma_\Delta^{(n)}$,
where $\Gamma_\Delta^{(0)}:=\{\varnothing\}$ and for $n\in\mathbb
N$, $\Gamma_\Delta^{(n)}$ consists of all $n$-point configurations in
$\Delta$.
Denote
\[
\widetilde{\Delta}^n:=\bigl\{(x_1,\ldots,x_n)
\in\Delta^n\mid\mbox{$x_i\ne x_j$ if $i\ne
j$}\bigr\}.
\]
Let $\mathcal B(\Gamma_\Delta^{(n)})$ denote the image of the $\sigma
$-algebra $\mathcal B(\widetilde{\Delta}^n)$ under the mapping
\[
\widetilde{\Delta}^n\ni(x_1,\ldots,x_n)
\mapsto\{x_1,\ldots,x_n\}\in\Gamma^{(n)}_\Delta.
\]
Then $\mathcal B(\Gamma_\Delta)$ is the minimal $\sigma$-algebra on
$\Gamma_\Delta$ which contains all $\mathcal B(\Gamma_\Delta^{(n)})$,
$n\in\mathbb N$.
A point process $\mu$ on $X$ has local densities in $\Delta$ if, for
each $n\in\mathbb N$,
there exists a nonnegative measurable symmetric function $d_\mu
^{(n)}[\Delta](x_1,\ldots,x_n)$ on $\widetilde{\Delta}^n$ such that
\begin{eqnarray*}
&&\int_{\Gamma_\Delta^{(n)}}f^{(n)}(\gamma) \mu_\Delta(d\gamma
)\\
&&\qquad=\frac1{n!}\int_{\widetilde{\Delta}^n}f^{(n)} \bigl(
\{x_1,\ldots,x_n\}\bigr)\,d_\mu^{(n)}[
\Delta](x_1,\ldots,x_n) m(dx_1)\cdots
m(dx_n)
\end{eqnarray*}
for each measurable function $f^{(n)}\dvtx \Gamma_\Delta^{(n)}\to
[0,\infty)$. We also denote $d_\mu^{(0)}[\Delta]:=\mu_\Delta(\{
\varnothing\})$. In the case where $X=\Delta$ (so that $X$ is a
compact Polish space), we will write $d_{\mu}^{(n)}$ instead of $d_\mu
^{(n)}[\Delta]$.

\begin{theorem}\label{jkfds5t}
Let $K\in\mathscr L (L^2(X,m))$ be $J$-self-adjoint. Let $K$ be a
locally trace-class operator on $X_1$ and $X_2$, and let $0\le\widehat
K\le1$. Then there exists a unique point process $\mu$ on $X$ which
has correlation functions
%
\begin{equation}\label{uydr}
k_\mu^{(n)}(x_1,\ldots,x_n)=\det
\bigl[K(x_i,x_j) \bigr]_{i,j=1,\ldots,n}.
\end{equation}
The Bogoliubov functional of $\mu$ is given by
%
\begin{equation}
\label{jcy} B_\mu(\varphi)=\Det\bigl(1+\operatorname{sgn}(\varphi)
\sqrt{|\varphi|} K\sqrt{|\varphi|} \bigr),\qquad \varphi\in B_0(X).
\end{equation}
If additionally $\|K\|<1$, then for each $\Delta\in\mathcal B_0(X)$,
the point process $\mu$ has local densities in $\Delta$:
%
\begin{eqnarray}\label{fdjgy}
d_\mu^{(0)}[\Delta]&=&\Det\bigl(1-K^\Delta\bigr),
\nonumber\\[-8pt]\\[-8pt]
d_\mu^{(n)}[\Delta](x_1,\ldots,x_n)&=&
\Det\bigl(1-K^\Delta\bigr)\det\bigl[L[\Delta](x_i,x_j)
\bigr]_{i,j=1,\ldots,n},\nonumber
\end{eqnarray}
where $L[\Delta]:=K^\Delta(1-K^\Delta)^{-1}$.
\end{theorem}

\begin{pf} By Proposition~\ref{kfdscd}, $\|K\|\le1$. We first
assume that $\|K\|<1$. We fix any compact $\Delta\subset X$. By
Proposition~\ref{fde6ewst}, $K^\Delta\in\mathscr L_{1|2}(L^2(X,m))$,
hence, $K^\Delta\in\mathscr L_{1|2}(L^2(\Delta,m))$. Clearly,
$K^\Delta$ is $J$-self-adjoint. Setting $\Delta_i:=\Delta\cap X_i$,
$i=1,2$, we get
%
\begin{eqnarray}\label{gf64w5q6}
P^\Delta\widehat K P^\Delta&=& P^\Delta
\bigl(KP_1+(1-K)P_2\bigr)P^{\Delta
}
\nonumber\\[-8pt]\\[-8pt]
&=& K^\Delta P^{\Delta_1}+\bigl(1-K^\Delta
\bigr)P^{\Delta_2}=\widehat{K}^\Delta,\nonumber
\end{eqnarray}
where the latter operator is understood as the transformation (\ref
{hsst}) of the operator $K^\Delta$ in the Hilbert space
$L^2(\Delta,m)=L^2(\Delta_1,m)\oplus L^2(\Delta_2,m)$. As $0\le\widehat
K\le1$, we conclude from (\ref{gf64w5q6}) that $0\le\widehat{K}^\Delta
\le1$. Since $\|K\|<1$, we have $\|K^\Delta\|<1$. Hence, by
Proposition~\ref{lkfdsa}, $\Det(1-K^\Delta)>0$.

Furthermore, by Proposition~\ref{fyd6s6yc}, the operator $L[\Delta]$
is $J$-self-adjoint and
\[
L[\Delta]\in\mathscr L_{1|2}\bigl(L^2(\Delta,m)\bigr),\qquad L[
\Delta]_{11}\ge0,\qquad L[\Delta]_{22}\ge0.
\]
Hence, we can choose an integral kernel $L[\Delta](x,y)$ of the
operator $L[\Delta]$ according to Proposition~\ref{hcdrt}. Therefore,
for any $x_1,\ldots,\break x_n\in\Delta_1$, $x_{n+1},\ldots,x_{n+m}\in\Delta_2$,
the matrix $ [L[\Delta](x_i,x_j) ]_{i,j=1,\ldots,n+m}$ is $J$-Hermitian
and the diagonal blocks
\[
\bigl[L[\Delta](x_i,x_j) \bigr]_{i,j=1,\ldots,n},\qquad
\bigl[L[\Delta](x_i,x_j) \bigr]_{i,j=n+1,\ldots,n+m}
\]
are nonnegative definite. Hence, by
Proposition~\ref{huydsd},
\[
\det\bigl[L[\Delta](x_i,x_j) \bigr]_{i,j=1,\ldots,n+m}
\ge0.
\]
Therefore, for each $n\in\mathbb N$, the function
\[
\widetilde{\Delta}^n\ni(x_1,\ldots,x_n)
\mapsto\det\bigl[L[\Delta](x_i,x_j)
\bigr]_{i,j=1,\ldots,n}
\]
is symmetric and takes nonnegative values.\vadjust{\goodbreak}

Hence, we can define a positive measure $\mu_\Delta$ on $(\Gamma_\Delta
,\mathcal B(\Gamma_\Delta))$ for which
%
\begin{eqnarray}\label{ufdrt6sw}
d_{\mu_\Delta}^{(0)}&=&\Det\bigl(1-K^\Delta\bigr),
\nonumber\\[-8pt]\\[-8pt]
\qquad d_{\mu_\Delta}^{(n)}(x_1,\ldots,x_n)&=&\Det
\bigl(1-K^\Delta\bigr)\det\bigl[L[\Delta](x_i,x_j)
\bigr]_{i,j=1,\ldots,n},\qquad n\in\mathbb N.\nonumber
\end{eqnarray}
Note that
\[
\det\bigl[L[\Delta](x_i,x_j) \bigr]_{i,j=1,\ldots,n}=0
\qquad\mbox{for all }(x_1,\ldots,x_n)\in\Delta^n
\setminus\widetilde{\Delta}_n,\qquad n\in\mathbb N.
\]
Hence, the Bogoliubov functional of $\mu_\Delta$ is given by
%
\begin{eqnarray}\label{vcdvgtd}
B_{\mu_\Delta}(\varphi)&=&\Det\bigl(1-K^\Delta\bigr)\nonumber\\
&&{}\times \Biggl(1+\sum
_{n=1}^\infty\frac1{n!}\int
_{\Delta^n} \bigl(1+\varphi(x_1)\bigr)\cdots\bigl(1+
\varphi(x_n)\bigr)
\nonumber\\[-8pt]\\[-8pt]
&&\hspace*{81.6pt}{} \times\det\bigl[L[\Delta](x_i,x_j)
\bigr]_{i,j=1,\ldots,n} m(dx_1)\cdots m(dx_n)
\Biggr),\nonumber\\
&&\eqntext{\varphi\in B(\Delta).}
\end{eqnarray}
Here $B(\Delta)$ denotes the set of all bounded measurable functions
on $\Delta$.
It follows from Proposition~\ref{bvgl} and (\ref{vcdvgtd}) that
%
\begin{equation}
\label{setfse} B_{\mu_\Delta}(\varphi)=\Det\bigl(1-K^\Delta\bigr)\Det
\bigl(1+L[\Delta](1+\varphi)\bigr), \qquad\varphi\in B(\Delta).
\end{equation}
Hence, by~\cite{BOO}, Corollary A.3, and Proposition~\ref{bvgl},
%
\begin{eqnarray}\label{trsa5}
B_{\mu_\Delta}(\varphi)&=&\Det\bigl(1-K^\Delta\bigr) \bigl(1+L[\Delta
](1+\varphi) \bigr)
\nonumber
\\
&=&\Det\bigl(1+K^\Delta\varphi\bigr)
\\
&=&\Det\bigl(1+\operatorname{sgn}(\varphi)\sqrt{|\varphi|}K\sqrt
{|\varphi|}
\bigr),\qquad \varphi\in B(\Delta).\nonumber
\end{eqnarray}

Now we take any sequence of compact subsets of $X$, $\{\Delta_n\}
_{n=1}^\infty$, such that
\[
\Delta_1\subset\Delta_2\subset\cdots,\qquad \bigcup
_{n=1}^\infty\Delta_n=X.
\]
By (\ref{trsa5}), the probability measures $\mu_{\Delta_n}$ on
$(\Gamma,\mathcal B_{\Delta_n}(\Gamma))$ form a consistent family of
probability measures. Therefore, by the Kolmogorov theorem, there
exists a unique probability measure on $(\Gamma,\mathcal B(\Gamma))$
such that the restriction of $\mu$ to each $\mathcal B_{\Delta
_n}(\Gamma)$ coincides with $\mu_{\Delta_n}$. By (\ref{trsa5}), the
Bogoliubov functional of $\mu$ is given by (\ref{jcy}), while the
statement about the local densities of $\mu$ follows from (\ref{ufdrt6sw}).
The determinantal form of the correlation functions of $\mu$---formula
(\ref{uydr})---follows from (\ref{gsts}), (\ref{jcy}) and
Proposition~\ref{bvgl}.

Let us now consider the case where $\|K\|=1$. For each $\varepsilon\in
(0,1)$, set $K_\varepsilon:=\varepsilon K$. Hence, $\|K_\varepsilon\|
<1$. We have
%
\begin{equation}
\label{jfd6}
\widehat K_\varepsilon= \varepsilon KP_1+(1-
\varepsilon K)P_2=\varepsilon\widehat K+(1-\varepsilon)P_2.
\end{equation}
Since $\widehat K\ge0$ and $P_2\ge0$, we get $ \widehat K_\varepsilon
\ge0$, and since $\widehat K\le1$ and $P_2\le1$, we get $ \widehat
K_\varepsilon\le1$. Hence, by the proved above, there exists a point
process $\mu_\varepsilon$ which has correlation functions
%
\begin{equation}
\label{sar4} k_{\mu_\varepsilon}^{(n)}(x_1,\ldots,x_n)=\varepsilon^n\det
\bigl[K(x_i,x_j) \bigr]_{i,j=1,\ldots,n}.
\end{equation}
Hence, the corresponding correlation measure is $\star$-positive
definite in the sense of~\cite{KK}; see also~\cite{LM}. By taking the
limit as $\varepsilon\to0$, we therefore conclude that the functions
%
\begin{equation}
\label{oie6w} k_\mu^{(n)}(x_1,\ldots,x_n):=\det\bigl[K(x_i,x_j)
\bigr]_{i,j=1,\ldots,n},\qquad n\in\mathbb N,
\end{equation}
determine a $\star$-positive definite correlation measure. By
Proposition~\ref{bvgl}, for each $\Delta\in\mathcal B_0(X)$ and $C>0$,
\[
1+\sum_{n=1}^\infty\frac{C^n}{n!}\int
_{\Delta^n}k_\mu^{(n)}(x_1,\ldots,x_n) m(dx_1)\cdots m(dx_n)=\Det
\bigl(1+CK^\Delta\bigr)<\infty.
\]
Hence, by~\cite{LM}, Corollary 1, we conclude that there exists a
unique probability measure $\mu$ on $(\Gamma,\mathcal B(\Gamma))$
which has correlation functions (\ref{oie6w}). By Proposition \ref
{bvgl} and formula (\ref{gsts}), the Bogoliubov functional of $\mu$
is given by (\ref{jcy}).
\end{pf}

The following corollary easily follows from Theorem~\ref{jkfds5t} and
Proposition~5.1 in~\cite{BO4} and its proof.

\begin{corollary}\label{ghydrd}
Let $G\dvtx  L^2(X_1,m)\to L^2(X_2,m)$ be a bounded linear operator such
that, for any $\Delta_1\in\mathcal B_0(X_1)$ and
$\Delta_2\in\mathcal B_0(X_2)$, the operators $GP^{\Delta_1}$ and
$P^{\Delta_2}G$ are Hilbert--Schmidt.
Let an operator $L\in\mathscr L(L^2(X,m))$ be defined by
\[
L:=\lleft[ %
\matrix{0&G
\cr
-G^*&0 } %
\rright].
\]
Then operator $1+L$ is invertible, and we set $K:=L(1+L)^{-1}$. We
further have the following:

\begin{longlist}
\item
The operator $K$ is $J$-self-adjoint.

\item The operator $K$ is locally trace-class on $X_1$ and $X_2$.

\item The operator $\widehat K$ is the orthogonal projection of
$L^2(X,m)$ onto the subspace
\[
\bigl\{h\oplus Gh\mid h\in L^2(X_1,m) \bigr\}.
\]

Thus, by Theorem~\ref{jkfds5t}, there exists a unique determiminantal
point process with correlation kernel $K(x,y)$.
\end{longlist}
\end{corollary}

\begin{remark}
As we mentioned in Section~\ref{vhytd},
the Whittaker kernel~\cite{BO2}, the matrix tail kernel~\cite{O}
and the continuous hypergeometric kernel~\cite{BO3} have their $L$
operators as in Corollary~\ref{ghydrd}, and so their $\widehat K$
operators are orthogonal projections.\vadjust{\goodbreak}
\end{remark}

\begin{pf*}{Proof of Corollary~\ref{ghydrd}}
That the operator $1+L$
is invertible is shown in~\cite{BO4}, Section 5. Statement (iii) is
just~\cite{BO4}, Proposition 5.1. By the proof of Proposition 5.1
\cite{BO4}, the operator $L$ has the following block form:
\begin{eqnarray*}
K_{11}&=&G^*G\bigl(1+G^*G\bigr)^{-1},
\\
K_{22}&=&GG^*\bigl(1+GG^*\bigr)^{-1},
\\
K_{21}&=&G\bigl(1+G^*G\bigr)^{-1},
\\
K_{21}&=&-G^*\bigl(1+G^*G\bigr)^{-1}.
\end{eqnarray*}
Hence, statement (i) obviously follows. So we only need to prove
statement (ii). To this end, we fix any $\Delta_1\in\mathcal
B_0(X_1)$ and
$\Delta_2\in\mathcal B_0(X_2)$.
By the assumption of the corollary,
$P^{\Delta_2}G$ is a Hilbert--Schmidt operator. Therefore,
\[
P^{\Delta_2}K_{21}P^{\Delta_1}=P^{\Delta_2}G\bigl(1+G^*G
\bigr)^{-1}P^{\Delta_1}
\]
is a Hilbert--Schmidt operator, hence so is the operator $ P^{\Delta
_1}K_{12}P^{\Delta_2}$. Again by the assumption of the corollary,
$GP^{\Delta_1}$ is a Hilbert--Schmidt operator,~hence,\looseness=-1
\[
\bigl(GP^{\Delta_1}\bigr)^*\bigl(GP^{\Delta_1}\bigr)=P^{\Delta
_1}G^*GP^{\Delta_1}
\]
is a trace-class operator.
Let $\{e^{(n)}\}_{n\ge1}$ be an orthonormal basis in $ L^2(\Delta
_1,m)$. Then, by the spectral theorem,
\begin{eqnarray*}
\sum_{n\ge1}\bigl(K_{11}
e^{(n)},e^{(n)}\bigr)_{L^2(\Delta_1,m)}&=&\sum
_{n\ge
1}\bigl(G^*G\bigl(1+G^*G\bigr)^{-1}
e^{(n)},e^{(n)}\bigr)_{L^2(\Delta_1,m)}
\\
&\le&\sum_{n\ge1}\bigl(G^*G e^{(n)},e^{(n)}
\bigr)_{L^2(\Delta_1,m)}<\infty.
\end{eqnarray*}
Therefore, the operator $P^{\Delta_1}K_{11}P^{\Delta_1}$ is
trace-class. Analogously, we may also show that the operator
$P^{\Delta_2}K_{22}P^{\Delta_2}$ is trace-class. Thus, statement (ii)
is proven.
 \end{pf*}

\begin{corollary} \label{dtsw}
Let an operator $K\in\mathscr L (L^2(X,m))$ be $J$-self-adjoint and
locally trace-class on $X_1$ and $X_2$. Let $0\le\widehat K\le1$ and
let $\|K\|=1$. Let $\mu$ be the corresponding determinantal point
process. Assume that $\Delta\in\mathcal B_0(X)$ is such that $\|
K^\Delta\|=1$. Then
\[
\mu_\Delta\bigl(\{\varnothing\}\bigr)=\Det\bigl(1-K^\Delta
\bigr)=0,
\]
that is, the $\mu$ probability of the event that there are no
particles in $\Delta$ is equal to zero.
\end{corollary}

\begin{pf} By (\ref{jcy}), for each $\Delta\in\mathcal B_0(X)$
and $z>0$,
\begin{eqnarray*}
\int_\Gamma e^{-z|\gamma\cap\Delta|} \mu(d\gamma)&=&\int
_\Gamma\prod_{x\in\gamma}\bigl(1+
\bigl(e^{-z}-1\bigr)\chi_\Delta\bigr) \mu(d\gamma)
\\
&=&\Det\bigl(1-\bigl(1-e^{-z}\bigr)K^\Delta\bigr).\vadjust{\goodbreak}
\end{eqnarray*}
Letting $z\to\infty$ and using the dominated convergence theorem, we get
\[
\mu_\Delta\bigl(\{\varnothing\}\bigr)=\Det\bigl(1-K^\Delta
\bigr).
\]
Since $\|K^\Delta\|=1$, by Proposition~\ref{ctehswt}, at least one of
the operators $K^{\Delta_1}=K^{\Delta_1}_{11}$, $K^{\Delta
_2}=K^{\Delta_2}_{22}$ must have norm 1. (Here, as above, $\Delta
_i=\Delta\cap X_i$, $i=1,2$.) Assume $\|K^{\Delta_1}\|=1$ (the other
case is analogous). As $\Det(1-K^{\Delta_1})$ is a classical Fredholm
determinant and the operator $K^{\Delta_1}$ is self-adjoint, we get
$\Det(1-K^{\Delta_1})=0$. Thus, we have $\mu_{\Delta_1}(\{
\varnothing\})=0$, that is, the $\mu$ probability of the event that
there are no particles in the set $\Delta_1$ is equal to 0. From here
the statement follows.
\end{pf}

\begin{remark}
Note that, for a determinantal point process $\mu$ with a
$J$-self-adjoint correlation operator $K$, the restriction of $\mu$ to
the $\sigma$-algebra $\mathcal B_{X_i}(\Gamma)$ ($i=1,2$) may be
identified with the determinantal point process on $X_i$ whose
correlation operator is the self-adjoint operator $K_{ii}$.
\end{remark}

We will now show that the conditions on a $J$-self-adjoint operator $K$
in Theorem~\ref{jkfds5t} are, in fact, necessary for a determinantal
point process with correlation kernel $K(x,y)$ to exist.

\begin{theorem}\label{hts}
Let $K\in\mathscr L(L^2(X,m))$ be $J$-self-adjoint, let
$K_{11}\ge0$,\break
$K_{22}\ge0$, and let $K^\Delta\in\mathscr L_{1|2}(L^2(X,m))$ for
each $\Delta\in\mathcal B_0(X)$.
Let an integral kernel $K(x,y)$ of the operator $K$ be chosen so that
statements \textup{(i)--(iii)} of Proposition~\ref{hcdrt} are satisfied.
Then there exists a unique point process $\mu$ on $X$ which has
correlation functions
(\ref{uydr}) if and only if $0\le\widehat K\le1$.
\end{theorem}

\begin{pf}
We only have to prove that, if a point process $\mu$ exists, then
$0\le\widehat K\le1$. We divide the proof into several steps.

(1) Fix any compact $\Delta\subset X$. By Proposition~\ref{bvgl} and
(\ref{gsts}), the Bogoliubov functional of $\mu$ is given by (\ref{jcy}).
Hence,
analogously to the proof of Corollary~\ref{dtsw}, we get
\[
\mu_\Delta\bigl(\{\varnothing\}\bigr)=\Det\bigl(1-K^\Delta
\bigr).
\]
In particular,
%
\begin{equation}
\label{vydhufty}\Det\bigl(1-K^\Delta\bigr)\ge0.
\end{equation}

(2) From now on we will additionally assume that $\|K\|<1$.
Then \mbox{$\|K^\Delta\|<1$} and we set $L[\Delta]:=K^\Delta(1-K^\Delta
)^{-1}$. Just\vspace*{1pt} as in the proof of Proposition~\ref{fyd6s6yc}, we derive
that $L[\Delta]$ is $J$-self-adjoint and $L[\Delta]\in\mathscr
L_{1|2}(L^2(X,m))$. To choose an integral kernel of the operator
$L[\Delta]$, we represent it in the form
\[
L[\Delta]=K^\Delta+K^\Delta L[\Delta].
\]
As $L[\Delta]\in\mathscr L_2(L^2(\Delta,m))$, we first choose
an arbitrary $J$-Hermitian integral kernel of this operator, which we
denote by $\widetilde L[\Delta](x,y)$.
Now we set
\[
L[\Delta](x,y):=K(x,y)+\int_\Delta K(x,z)\widetilde L[
\Delta](z,y) m(dy),\qquad x,y\in\Delta.
\]
As is easily seen, this integral kernel satisfies
%
\begin{eqnarray}
\label{jkgfuyfu}
&&\Tr\bigl(L[\Delta]_{\mathrm
{even}}^\Lambda\bigr)\nonumber\\[-8pt]\\[-8pt]
&&\qquad=\int
_\Lambda L[\Delta](x,x) m(dx) \qquad\mbox{for each } \Lambda\in\mathcal
B_0(X), \Lambda\subset\Delta.\nonumber
\end{eqnarray}

(3) By (\ref{jkgfuyfu}),
%
\begin{equation}
\label{cfysdtrd}\Tr\bigl(\bigl(L[\Delta](1+\varphi)\bigr)_{\mathrm{even}}
\bigr)=\int_\Delta L[\Delta](x,x) \bigl(1+\varphi(x)\bigr),\qquad
\varphi\in B_0(X).
\end{equation}
Since formula (\ref{setfse}) clearly holds for the Boliubov functional
of $\mu_\Delta$, using (\ref{cfysdtrd}) and Proposition~\ref{fda},
we get, for each $\varphi\in B(\Delta)$,
\begin{eqnarray*}
B_{\mu_\Delta}(\varphi)&=&\Det\bigl(1-K^\Delta\bigr) \\
&&{}\times\Biggl(1+\sum
_{n=1}^\infty\int_{\Delta^n}
\bigl(1+\varphi(x_1)\bigr)\cdots\bigl(1+\varphi(x_n)
\bigr)
\\
&&\hspace*{69.2pt}{}\times\det\bigl[L[\Delta](x_i,x_j)
\bigr]_{i,j=1,\ldots,n} m(dx_1)\cdots m(dx_n) \Biggr).
\end{eqnarray*}
This implies that the measure $\mu_\Delta$ has densities (\ref
{ufdrt6sw}). Hence, by (\ref{vydhufty}), for
$m^{\otimes n}$-a.a. $(x_1,\ldots,x_n)\in\Delta^n$,
\[
\det\bigl[L[\Delta](x_i,x_j) \bigr]_{i,j=1,\ldots,n}
\ge0.
\]
In particular, for $i=1,2$, for
$m^{\otimes n}$-a.a. $(x_1,\ldots,x_n)\in\Delta_i^n$,
%
\begin{equation}
\label{bfytdy7} \det\bigl[L[\Delta]_{ii}(x_i,x_j)
\bigr]_{i,j=1,\ldots,n}\ge0.
\end{equation}
Here $\Delta_i:=\Delta\cap X_i$, $i=1,2$.

(4) Following~\cite{O}, Proposition 1.5, let us find a representation
of $L[\Delta]_{11}$ in terms of the blocks of the operator $K^\Delta
$. Since $L[\Delta](1-K^\Delta)=K^\Delta$, we have
%
\begin{eqnarray}
\label{dtrs}
L[\Delta]_{11}\bigl(1-K^\Delta_{11}\bigr)-L[
\Delta]_{12}K^\Delta_{21}&=&K^\Delta_{11},
\\
\label{fdts}
-L[\Delta]_{11}K^\Delta_{12}+ L[
\Delta]_{12}\bigl(1-K^\Delta_{22}\bigr)&=&K^\Delta_{12}.
\end{eqnarray}
From (\ref{fdts}),
\[
-L[\Delta]_{11}K^\Delta_{12}\bigl(1-K^\Delta_{22}
\bigr)^{-1}+ L[\Delta]_{12}=K^\Delta_{12}
\bigl(1-K^\Delta_{22}\bigr)^{-1},
\]
hence,
\[
-L[\Delta]_{11}K^\Delta_{12}\bigl(1-K^\Delta_{22}
\bigr)^{-1}K^\Delta_{21}+ L[\Delta]_{12}K^\Delta_{21}=K^\Delta_{12}
\bigl(1-K^\Delta_{22}\bigr)^{-1}K^\Delta_{21}.\vadjust{\goodbreak}
\]
Adding this to (\ref{dtrs}) yields
\[
L[\Delta]_{11}\bigl(1-Q[\Delta]_{11}\bigr)=Q[
\Delta]_{11},
\]
where
%
\begin{eqnarray}\label{vgftyd}
Q[\Delta]_{11}:\!&=& K^\Delta_{11}+K^\Delta_{12}
\bigl(1-K_{22}^\Delta\bigr)^{-1}K^\Delta_{21}
\nonumber\\[-8pt]\\[-8pt]
&=&K^\Delta_{11}-\bigl(K^\Delta_{21}\bigr)^*
\bigl(1-K_{22}^\Delta\bigr)^{-1}K^\Delta_{21}.\nonumber
\end{eqnarray}
Since the operator $1-K^\Delta_{11}$ is strictly positive and the
operator $(K^\Delta_{21})^*(1-K_{22}^\Delta)^{-1}K^\Delta_{21}$ is
nonnegative, the operator $1-Q[\Delta]_{11}$ is strictly positive,
hence invertible. Therefore,
%
\begin{equation}
\label{cydyrjzxfzgg} L[\Delta]_{11}=Q[\Delta]_{11}\bigl(1-Q[
\Delta]_{11}\bigr)^{-1}.
\end{equation}

(5) By (\ref{vgftyd}), the operator $Q[\Delta]_{11}$ is self-adjoint
and trace-class. Since the operator $1-Q[\Delta]_{11}$ is strictly
positive, we therefore get
\[
\Det\bigl(1-Q[\Delta]_{11}\bigr)>0.
\]
[Note that $\Det(1-Q[\Delta]_{11})$ is a usual Fredholm determinant.]
Therefore, by~(\ref{bfytdy7}), we can define a nonnegative, finite
measure $\nu[\Delta_1]$ on $(\Gamma_{\Delta_1},\mathcal B(\Gamma_{\Delta
_1}))$ whose local densities are
%
\begin{eqnarray}\label{ufJHXFGUaf}
d_{\nu[\Delta_1]}^{(0)}&=&\Det\bigl(1-Q[\Delta]_{11}\bigr),
\nonumber\\
d_{\nu[\Delta_1]}^{(n)}(x_1,\ldots,x_n)&=&\Det
\bigl(1-Q[\Delta]_{11}\bigr)\det\bigl[L[\Delta]_{11}(x_i,x_j)
\bigr]_{i,j=1,\ldots,n},\\
&&\eqntext{n\in\mathbb N.}
\end{eqnarray}
Analogously to (\ref{ufdrt6sw})--(\ref{trsa5}), we conclude from
(\ref{ufJHXFGUaf}) that the Bogoliubov transform of the measure $\nu
[\Delta_1]$ is given by
%
\begin{eqnarray}\label{jftydygcx}
B_{\nu[\Delta_1]}(\varphi)&=&\int_{\Gamma_{\Delta_1}}\prod
_{x\in
\gamma}\bigl(1+\varphi(x)\bigr) \nu[\Delta_1](d
\gamma)
\nonumber\\[-8pt]\\[-8pt]
&=&\Det\bigl(1+\operatorname{sgn}(\varphi)\sqrt{|\varphi|} Q[
\Delta]_{11}\sqrt{|\varphi|} \bigr),\qquad \varphi\in B_0(
\Delta_1).\nonumber
\end{eqnarray}
Setting $\varphi\equiv0$, we see that $\nu[\Delta_1](\Gamma_{\Delta
_1})=1$, that is, $\nu[\Delta_1]$ is a point process in
$\Delta_1$.

(6) We can now choose an integral kernel of the operator $Q[\Delta
]_{11}$ analogously to~\cite{GY}, Lemma A.3, and~\cite{LM}, Section 3.
Indeed, since $K_{21}^\Delta$ is a Hilbert--Schmidt operator,
$(K^\Delta_{21})^*(1-K_{22}^\Delta)^{-1}K^\Delta_{21}$ is a
nonnegative trace-class operator in $L^2(\Delta_1,m)$. The operator
$ ((K^\Delta_{21})^*(1-K_{22}^\Delta)^{-1}K^\Delta_{21}
)^{1/2}$ is Hilbert--Schmidt, hence an integral operator. We choose its
integral kernel, denoted by $\theta(x,y)$, so that
\begin{eqnarray*}
\theta(x,y)&=&\overline{\theta(y,x)} \qquad\mbox{for all
$x,y\in\Delta_1$},\\
\theta(x,\cdot)&\in& L^2(\Delta_1,m) \qquad\mbox{for all $x\in
\Delta_1$}.
\end{eqnarray*}
[Recall that $\int_{\Delta_1^2}|\theta(x,y)|^2 m(dx) m(dy)=
\| ((K^\Delta_{21})^*(1-K_{22}^\Delta)^{-1}K^\Delta_{21} )^{1/2} \|
_2^2<\infty$.] Now, we set an integral kernel
of the operator $(K^\Delta_{21})^*(1-K_{22}^\Delta)^{-1}K^\Delta_{21}$
to be
\begin{eqnarray*}
\bigl(K^\Delta_{21}\bigr)^*\bigl(1-K_{22}^\Delta
\bigr)^{-1}K^\Delta_{21}(x,y):\!&=&\int
_{\Delta_1}\theta(x,z)\theta(z,y) m(dz)
\\
&=&\bigl(\theta(x,\cdot),\theta(y,\cdot)\bigr)_{L^2(\Delta_1,m)},\qquad x,y\in
\Delta_1.
\end{eqnarray*}
We similarly construct an integral kernel of the operator
$K_{11}^\Delta$:
\[
K_{11}^\Delta(x,y)=\bigl(\eta(x,\cdot),\eta(y,\cdot)
\bigr)_{L^2(\Delta
_1,m)},\qquad x,y\in\Delta_1.
\]
Hence, by virtue of (\ref{vgftyd}), we may choose an integral kernel
of the operator $Q[\Delta]_{11}$ as follows:
%
\begin{equation}
\label{fxtdst}\quad Q[\Delta]_{11} (x,y)= \bigl(\eta(x,\cdot),\eta(y,\cdot)
\bigr)_{L^2(\Delta_1,m)}- \bigl(\theta(x,\cdot),\theta(y,\cdot)
\bigr)_{L^2(\Delta_1,m)}.
\end{equation}
As is easily seen, for each $\Delta\in\mathcal B_0(X)$, $\Lambda\subset
\Delta_1$,
\[
\Tr\bigl(Q[\Delta]_{11}^\Lambda\bigr)=\int_\Lambda
Q[\Delta]_{11}(x,x) m(dx).
\]

Now, analogously to Proposition~\ref{bvgl}, we get from (\ref{jftydygcx})
\begin{eqnarray*}
B_{\nu[\Delta_1]}(\varphi)&=& 1+\sum_{n=1}^\infty
\frac1{n!}\int_{\Delta_1^n}\varphi(x_1)\cdots
\varphi(x_n)\\
&&\hspace*{63pt}{}\times\operatorname{det} \bigl[Q[\Delta]_{11}(x_i,x_j)
\bigr]_{i,j=1,\ldots,n} m(dx_1)\cdots m(dx_n)
\end{eqnarray*}
for each $\varphi\in B_0(\Delta_1)$. Hence, the correlation functions
of the point process $\nu[\Delta_1]$ are
\[
k_{\nu[\Delta_1]}^{(n)}(x_1,\ldots,x_n)=
\operatorname{det} \bigl[Q[\Delta]_{11}(x_i,x_j)
\bigr]_{i,j=1,\ldots,n},\qquad n\in\mathbb N.
\]
Therefore, for each $n\in\mathbb N$,
%
\begin{equation}
\label{ghfdtydy}\quad \operatorname{det} \bigl[Q[\Delta]_{11}(x_i,x_j)
\bigr]_{i,j=1,\ldots,n}\ge0 \qquad\mbox{for $m^{\otimes n}$-a.a. $(x_1,\ldots,x_n)\in\Delta_1^n$}.
\end{equation}

(7) Obviously, the following two mappings are measurable:
\[
\Delta_1\ni x\mapsto\eta(x,\cdot)\in L^2(
\Delta_1,m),\qquad \Delta_1\ni x\mapsto\theta(x,\cdot)\in
L^2(\Delta_1,m).
\]
Therefore, by Lusin's theorem (see, e.g.,~\cite{Schwartz}), for each
$\varepsilon>0$, there exists a compact set $\Lambda_\varepsilon
\subset\Delta_1$
such that $m(\Delta_1\setminus\Lambda_\varepsilon)\le\varepsilon$
and the mappings
\[
\Lambda_\varepsilon\ni x\mapsto\eta(x,\cdot)\in L^2(
\Delta_1,m), \qquad\Lambda_\varepsilon\ni x\mapsto\theta(x,\cdot)\in
L^2(\Delta_1,m)
\]
are continuous. Therefore, by (\ref{fxtdst}), the function
%
\begin{equation}
\label{vgcdtyd}\Lambda_\varepsilon^2\ni(x,y)\mapsto Q[
\Delta]_{11}(x,y)\in\mathbb C
\end{equation}
is continuous. Hence, by (\ref{ghfdtydy}),
\[
\operatorname{det} \bigl[Q[\Delta]_{11}(x_i,x_j)
\bigr]_{i,j=1,\ldots,n}\ge0 \qquad\mbox{for all $(x_1,\ldots,x_n)
\in\Lambda_\varepsilon^n$}.
\]
Thus, the continuous kernel (\ref{vgcdtyd}) is positive definite, and
therefore the operator $Q[\Delta]_{11}$ is nonnegative on
$L^2(\Lambda_\varepsilon,m)$. By letting $\varepsilon\to0$, we
conclude that
$Q[\Delta]_{11}\ge0$ on $L^2(\Delta_1,m)$. Hence, by (\ref{vgftyd}),
%
\begin{equation}
\label{kjftydrk} K^\Delta_{11}\ge\bigl(K^\Delta_{21}
\bigr)^*\bigl(1-K_{22}^\Delta\bigr)^{-1}K^\Delta_{21}
\qquad\mbox{on }L^2(\Delta_1,m).
\end{equation}

(8) We denote by $\widehat{K}^\Delta$
the corresponding transformation of the operator $K^\Delta$
in the Hilbert space $L^2(\Delta,m)=L^2(\Delta_1,m)\oplus L^2(\Delta
_2,m)$. Hence, $\widehat{K}^\Delta=P^\Delta\widehat K P^\Delta$ and
\[
\widehat{K}^\Delta=\lleft[ %
\matrix{
K^\Delta_{11}& K^\Delta_{21}
\vspace*{2pt}\cr
\bigl(K^\Delta_{21}\bigr)^*&1-K^\Delta_{22} }
\rright].
\]
By (\ref{kjftydrk}), for each $f=(f_1,f_2)\in L^2(\Delta,m)$,
%
\begin{eqnarray}\label{dw46}
\bigl(\widehat{K}^\Delta f,f\bigr)&=&\bigl(K_{11}^\Delta
f_1,f_1\bigr)+\bigl(K_{21}^\Delta
f_1,f_2\bigr)\nonumber\\
&&{}+\bigl(\bigl(K_{21}^\Delta
\bigr)^* f_2,f_1\bigr)+\bigl(\bigl(1-K^\Delta_{22}
\bigr)f_2,f_2\bigr)
\nonumber
\\
&\ge&\bigl(\bigl(K^\Delta_{21}\bigr)^*\bigl(1-K_{22}^\Delta
\bigr)^{-1}K^\Delta_{21}f_1,f_1
\bigr)\nonumber\\[-8pt]\\[-8pt]
&&{}+\bigl(\bigl(1-K^\Delta_{22}\bigr)f_2,f_2
\bigr)-2\bigl|\bigl(K_{21}^\Delta f_1,f_2
\bigr)\bigr|
\nonumber
\\
&=&\bigl(\bigl(1-K_{22}^\Delta\bigr)^{-1}K^\Delta_{21}f_1,K^\Delta_{21}f_1
\bigr)\nonumber\\
&&{}+\bigl(\bigl(1-K^\Delta_{22}\bigr)f_2,f_2
\bigr)-2\bigl|\bigl(K_{21}^\Delta f_1,f_2
\bigr)\bigr|.\nonumber
\end{eqnarray}
Since $K_{22}^\Delta$ is a compact self-adjoint operator in
$L^2(\Delta_2,m)$, we can choose an orthnormal basis of $L^2(\Delta
_2,m)$ which consists of eigenvectors of the operator $K_{22}^\Delta$,
and we denote by $\lambda_n$ the eigenvalue belonging to eigenvector
$e_n$, $n\ge1$. Clearly, $\lambda_n<1$ for all $n$. Then, by (\ref{dw46}),
\begin{eqnarray*}
\bigl(\widehat{K}^\Delta f,f\bigr)&\ge&\sum_{n=1}^\infty(1-
\lambda_n)^{-1}\bigl|\bigl(K_{21}^\Delta
f_1,e_n\bigr)\bigr|^2 +\sum
_{n=1}^\infty(1-\lambda_n)\bigl|(f_2,e_n)\bigr|^2
\\
&&{} -\sum_{n=1}^\infty2 \bigl|\bigl(K_{21}^\Delta
f_1,e_n\bigr) (f_2,e_n)\bigr|
\\
&=&\sum_{n=1}^\infty\bigl( (1-
\lambda_n)^{-1/2}\bigl|\bigl(K_{21}^\Delta
f_1,e_n\bigr)\bigr| - (1-\lambda_n)^{1/2}\bigl|(f_2,e_n)\bigr|
\bigr)^2\ge0.
\end{eqnarray*}
Thus, for each compact $\Delta\subset X$, the operator $\widehat
{K}^\Delta=P_\Delta\widehat KP_\Delta$ is nonnegative. Hence,
$\widehat K\ge0$. Exchanging\vspace*{1pt} the role of the sets $X_1$ and $X_2$ and
using instead of the operator $K$ the operator $1-K$, we therefore get
$1-\widehat K\ge0$. Thus, $0\le\widehat K\le1$.

(9) We now assume that $\|K\|=1$. Using the procedure of thinning of
the point process $\mu$ (see, e.g.,~\cite{DVJ}, Example 8.2(a)), we
conclude that, for each $\varepsilon\in(0,1)$, there exists a point
process $\mu_\varepsilon$ which has correlation functions as in
formula (\ref{sar4}), that is, a determinantal point process\vspace*{2pt}
corresponding to the operator $K_\varepsilon:=\varepsilon K$. By the
proved above
$0\le\widehat K_\varepsilon\le1$. Hence, by (\ref{jfd6}),
\[
0\le\varepsilon\widehat K+(1-\varepsilon)P_2\le1.
\]
Letting $\varepsilon\to1$, we get $0\le\widehat K\le1$.

(10) Finally, we assume that $\|K\|>1$ and we have to show that a
determinantal point process does not exist in this case. Assume the
contrary, that is, assume that there exists a determinantal point
process with correlation kernel $K(x,y)$. Since $\|K\|>1$, there exists
a compact set $\Delta\subset X$ such that $\|K^\Delta\|>1$. Analogously
to part (9), using the procedure of thinning, we conclude that, for
each $\varepsilon\in(0,1)$, there exists a determinantal point process
with correlation kernel $K_\varepsilon(x,y):=\varepsilon K(x,y)$. We
choose $\varepsilon:=\|K^\Delta\|^{-1}$, so that
$\|K_\varepsilon^\Delta\|=1$. We take the restriction of the
corresponding probability measure to the $\sigma$-algebra $\mathcal
B_\Delta(\Gamma)$, that is, a point process on $(\Gamma_\Delta,\mathcal
B(\Gamma_\Delta))$. We denote this point process by
$\mu_{\varepsilon,\Delta}$. By part~(9), we have
$0\le\widehat{K}^\Delta_\varepsilon\le1$. Then, by Corollary
\ref{dtsw},
\[
\Det\bigl(1-K^\Delta_\varepsilon\bigr)=0.
\]

Next, following the idea of~\cite{Soshnikov}, Remark 4, we consider
\begin{eqnarray*}
\int_\Gamma(1-\varepsilon)^{|\gamma\cap\Delta|} \mu(d\gamma)&=&\int
_\Gamma\prod_{x\in\gamma}\bigl(1-
\varepsilon\chi_\Delta(x)\bigr) \mu(d\gamma)
\\
&=&\Det\bigl(1-\varepsilon K^\Delta\bigr)=\Det\bigl(1-
K_\varepsilon^\Delta\bigr)=0.
\end{eqnarray*}
On the other hand, $(1-\varepsilon)^{|\gamma\cap\Delta|}>0$ for all
$\gamma\in\Gamma$. Hence,
\[
\int_\Gamma(1-\varepsilon)^{|\gamma\cap\Delta|} \mu(d\gamma)>0,
\]
which is a contradiction.
\end{pf}
%

\section*{Acknowledgments}

I am extremely grateful to Grigori Olshanski for the formulation of
the problem and for many useful discussions and suggestions. I wish
also to thank Marek Bo\.zejko for a discussion on positive definite
kernels. I am grateful to the anonymous referee for several useful
suggestions.



\printaddresses

\end{document}